\documentclass[11pt,a4paper,oneside,reqno]{amsart}
\usepackage{amsmath,amsthm,amsfonts,amssymb,bm}
\usepackage{mathrsfs}
\usepackage{color}
\usepackage[dvipsnames]{xcolor}
\usepackage{graphicx}
\usepackage{comment}
\usepackage{commath}  
\usepackage{mathtools}

\makeatletter    
\def\l@section{\@tocline{1}{0,2pt}{2pc}{24mm}{\ \ }}

\usepackage[normalem]{ulem}
\usepackage[top=1.06in, bottom=1.06in, left=1.06in, right=1.06in]{geometry}
\usepackage[pdftex, unicode=true,
            colorlinks,
            linkcolor=blue,
            anchorcolor=blue,
            citecolor=blue
            ]{hyperref}

\usepackage{subcaption}
\usepackage{titlesec}
\titleformat{\section}{\vskip5pt\large\bfseries}{\thesection.}{0.5em}{\centering\vspace{3pt}}
\titleformat{\subsection}{\vskip5pt\normalsize\bfseries}{\thesubsection.}{0.5em}{}

\theoremstyle{definition}
\newtheorem{example}{Example}[section]
\theoremstyle{plain}
\newtheorem{theorem}{Theorem}
\newtheorem{lemma}{Lemma}

\newtheorem{proposition}{Proposition}
\newtheorem{remark}{Remark}[section]

\newtheorem{problem}{Problem}

\numberwithin{theorem}{section}
\numberwithin{proposition}{section}
\numberwithin{lemma}{section}
\numberwithin{equation}{section}

\usepackage{xcolor}

\definecolor{link}{rgb}{0.45,0.51,0.67}

\newcommand{\Ker}{\text{Ker}}

\newcommand{\Man}{\mathcal{M}}

\newcommand{\RM}{\mathrm{RM}}
\newcommand{\Vol}{\mathrm{Vol}}

\newcommand{\normt}[1]{%
	\vert\kern-0.9pt\vert\kern-0.9pt\vert #1
	\vert\kern-0.9pt\vert\kern-0.9pt\vert}

\allowdisplaybreaks

\begin{document}

\title[]{Finite element methods for isometric embedding of Riemannian manifolds} 

\author{Guangwei Gao$^*$}
\thanks{$^*$\,Department of Applied Mathematics, The Hong Kong Polytechnic University, Hong Kong.
E-mail: guang-wei.gao@polyu.edu.hk, buyang.li@polyu.edu.hk}

\author{Kaibo Hu$^\dagger$}
\thanks{$^\dagger$\,Mathematical Institute, University of Oxford, Oxford OX2 6GG, United Kingdom.
E-mail: kaibo.hu@maths.ox.ac.uk, ganghui.zhang@maths.ox.ac.uk}

\author{Buyang Li$^*$}

\author{Ganghui Zhang$^\dagger$}

\date{}

\begin{abstract}
The isometric embedding problem for Riemannian manifolds, which connects intrinsic and extrinsic geometry, is a central question in differential geometry with deep theoretical significance and wide-ranging applications. Despite extensive analytical progress, the nonlinear and degenerate nature of this problem has hindered the development of rigorous numerical analysis in this area. As the first step toward addressing this gap, we study the numerical approximation of Weyl’s problem, i.e., the isometric embedding of two-dimensional Riemannian manifolds with positive Gaussian curvature into $\mathbb{R}^3$, by establishing a new weak formulation that naturally leads to a numerical scheme well suited for high-order finite element discretization, and conducting a systematic analysis to prove the well-posedness of this weak formulation, the existence and uniqueness of its numerical solution, as well as its convergence with error estimates. This provides a foundational framework for computing isometric embeddings of Riemannian manifolds into Euclidean space, with the goal of extending it to a broader range of cases and applications in the future. Our framework also extends naturally to the isometric embedding of the Ricci flow, with rigorous error estimates, enabling the visualization of geometric evolutions in intrinsic curvature flows. Numerical experiments support the theoretical analysis by demonstrating the convergence of the method and its effectiveness in simulating isometric embeddings of given Riemannian manifolds as well as Ricci flows.\\


\end{abstract}

\setlength\abovedisplayskip{4pt}
\setlength\belowdisplayskip{4pt}



\maketitle

\vspace{-15pt}

\setcounter{tocdepth}{1}
\tableofcontents

\section{Introduction}

Given a Riemannian manifold \( \Man \) with a target metric \( g \), the problem of isometric embedding is to find an embedding map \( r: \Man \rightarrow \mathbb{R}^N \) such that the pullback of the Euclidean metric under this map coincides with the target metric \( g \). This is a classical and fundamental question in differential geometry.
In 1956, Nash proved that any smooth $n$-dimensional Riemannian manifold admits a global smooth isometric embedding into some Euclidean space $\mathbb{R}^N$ \cite{nash1956imbedding}. Subsequent work by Günther reduced the required target dimension \cite{Guenther1989NashEmbedding}. It remains unclear whether the known bounds on $N$ are optimal. For two-dimensional manifolds, Weyl posed the problem of finding a global isometric embedding of a Riemannian metric with positive Gaussian curvature into $\mathbb{R}^3$ \cite{Weyl1916}. Lewy resolved the analytic case \cite{Lewy1938}, and Nirenberg independently established a complete solution for smooth metrics \cite{Nirenberg1953}. For a comprehensive review of isometric embeddings of Riemannian manifolds in Euclidean spaces, we refer the reader to \cite{han2006isometric}. Despite these significant theoretical achievements, a critical gap remains from the perspective of computational mathematics: all aforementioned existence proofs are non-constructive, providing no direct algorithm for computing the embedded surface numerically, hindering the ability to apply these results in the numerical settings. 

The demand for constructive and numerically robust methods for isometric embedding is driven by applications across a broad range of computational sciences. From a geometric perspective, isometric embedding bridges intrinsic and extrinsic viewpoints and enables the visualization of the evolving geometry of intrinsic geometric flows, such as the Ricci flow \cite{rubinstein2005visualizing}. In numerical relativity, isometric embedding is used to study the geometry of black hole horizons \cite{bondarescu2002isometric,tichy2014new, ray2015adiabatic, jasiulek2012isometric}. Moreover, many definitions of quasi-local energy in general relativity, including the Brown--York and Wang--Yau energies, require isometric embeddings of convex surfaces in $\mathbb{R}^3$ \cite{BrownYork1993, wang2009quasilocal}. In elasticity, surface models of thin shells are closely related to isometric embedding: the metric tensor encodes the strain tensor, while the embedding represents the deformation of the elastic body \cite{ciarlet2022mathematical, ciarlet2016korn}. In nonlinear bending models, isometric embedding arises as a constraint in the minimization of bending energies for nonlinear plates \cite{bartels2017bilayer}. In computer graphics, isometric embedding has applications in geometry processing and art-directed shape design \cite{chern2018shape}. At the discrete level, isometric embedding is also referred to as the Discrete Geometry Problem (DGP), with applications in molecular distance geometry, sensor network localization, and graph drawing \cite{liberti2011molecular}. Moreover, the isometric embedding problem is used to embed high-dimensional statistical data into a low-dimensional submanifold, closely related to multidimensional scaling (MDS) in statistics \cite{borg2005modern}.


Numerically approximating isometric embedding presents profound difficulties. Firstly, as demonstrated in~\eqref{eq:isoemb}, an isometric embedding satisfies a \emph{nonlinear} system of partial differential equations (PDEs). It possesses the property that any solution remains valid after composition with a rigid motion. This property, known as \emph{rigidity}, renders~\eqref{eq:isoemb} \emph{degenerate}.
For this nonlinear and degenerate PDE, formulating a weak formulation that is both mathematically well-posed and numerically amenable to discretization is far from straightforward.
Furthermore, the nonlinear and degenerate nature of the problem introduces significant challenges in the numerical analysis. This includes handling the degeneration at the discrete level, establishing the existence and uniqueness of numerical solutions, and performing a rigorous error and convergence analysis.

In~\cite{hotz2004isometric}, Hotz and Hagen proposed an algorithm for computing isometric embeddings by reconstructing the surface ring by ring using Euclidean distances. In~\cite{jasiulek2012isometric,tichy2014new}, spectral methods were employed to discretize the PDEs governing isometric embeddings. In~\cite{ray2015adiabatic,chern2018shape,bondarescu2002isometric}, the problem was reformulated as a minimization problem and solved via optimization algorithms. Moreover, a Delaunay triangulation algorithm is proposed to accommodate abstract Riemannian manifold in \cite{boissonnat2018delaunay}. However, despite these developments, a rigorous numerical analysis---including convergence and error estimates---of any method for any formulation of the isometric embedding problem has remained unavailable. This work represents a first step toward addressing this gap. 

Specifically, we study the numerical approximation for the isometric embedding of two-dimensional Riemannian manifolds with positive Gaussian curvature into \( \mathbb{R}^3 \). We establish a new weak formulation, which naturally leads to a numerical scheme well-suited for high-order finite element discretization. We conduct a systematic study to rigorously prove the well-posedness of this new weak formulation, the existence and uniqueness of its numerical approximation, as well as its convergence and error estimates.
This provides a foundational approach for this class of problems, with the aim of extending it to a broader range of cases and various applications in the future.

The starting point of our work is inspired by the continuity method for proving the existence of isometric embeddings of a metric with positive Gaussian curvature \cite{Nirenberg1953, han2006isometric}. This method first constructs a continuous family of metrics \( g(t) \), \( 0 \le t \le 1 \), connecting the given traget metric \( g \) to the induced metric of the unit sphere \( S^2 \) via the uniformization theorem, with all metrics \( g(t) \) maintaining positive Gaussian curvature. It is then shown that the set of parameters \( t \in [0,1] \) for which the metric \( g(t) \) admits an isometric embedding \( r(t): \Man \to \mathbb{R}^3 \) is both open and closed. As a result, the target metric \( g = g(1) \) necessarily admits an isometric embedding into \( \mathbb{R}^3 \). While this approach guarantees existence, it remains non-constructive. 
We propose a new formulation that transforms this theoretical analysis into a numerical method which solves the \emph{embedding flow} \( r(t) \) in the continuity method described above. Starting from a known initial embedding, the flow is evolved by solving a linearized system for the velocity \( \partial_t r(t) \), derived from the differentiating the isometric embedding equation. The initial embedding is then gradually deformed according to the velocity computed from this linearized system, ultimately leading to the final isometric embedding of the target metric.
Note that the continuity method transforms the problem of embedding a fixed metric into a sequence of dynamic embedding problems. In contrast, our approach further converts these nonlinear dynamic embedding problems in the continuity method into the task of solving a series of PDEs for the velocity of the flow. In this way, our method bridges the gap between non-constructive existence proofs and computational realizations of Riemannian embeddings. 
This evolution not only recovers the isometric embedding for a fixed metric but also naturally extends to the isometric embedding of evolving metrics, where a family of metrics \( g(t) \) serves as a solution to the intrinsic curvature flow, such as the Ricci flow \cite{gawlik2019finite,gao2025ricci, rubinstein2005visualizing}. In this context, the corresponding isometric embedding \( r(t) \) can be used to visualize the geometric evolution of the flow.



Notably, the degeneracy of the nonlinear isometric embedding problem manifests in the linearized system governing the velocity, causing the velocity equation to exhibit a nontrivial kernel corresponding to the space of \emph{infinitesimal rigid motions}. To address this issue, we propose a new variational formulation that solves for the velocity as the component orthogonal to the space of {infinitesimal rigid motions}, thereby uniquely determining the velocity. 
As a result, the new variational formulation in \eqref{eq:vel} leads to a saddle point system, and its well-posedness can be proven using the abstract theory in \cite{boffi2013mixed}. The key to the proof hinges on establishing a \emph{Korn-type inequality} on manifolds. In the literature, Korn-type inequalities on manifolds primarily arise in three contexts.
The first is in the study of surface Stokes~\cite{bonito2020divergence} and Navier--Stokes equations~\cite{jankuhn2018incompressible}, where only tangential vector fields are considered. The second context is found in pure Riemannian geometry, where intrinsic language is used~\cite{chen2002riemannian, duduchava2010lions}. The third, which is closest to our setting, concerns surface models of thin shells~\cite{ciarlet2016korn}, although it typically focuses on \emph{open} surfaces with positive Gaussian curvature.
These Korn-type inequalities also differ in the conditions imposed on the vector field; specifically, they differ in the set on which the vector field must vanish for a Korn inequality to hold.
In \cite{jankuhn2018incompressible, bonito2020divergence}, tangential vector fields are required to be orthogonal to the space of Killing vector fields; in \cite{ciarlet2016korn}, the tangential component vanishes on the boundary of an open surface; and in \cite{chen2002riemannian}, it suffices that they vanish on a subset of Hausdorff dimension greater than \(n-2\) of an \(n\)-dimensional manifold. 
None of these existing Korn-type inequalities are sufficient for our purposes. Therefore, in Lemma~\ref{lm:Korn}, we present an appropriate Korn inequality for two-dimensional closed Riemannian manifolds with positive Gaussian curvature. Our proof builds upon the study of the linearized equation for isometric embeddings in \cite[Lemma 9.2.2]{han2006isometric}, which is detailed in~\ref{appendix:Korn}.

The proposed variational formulation admits a natural finite element discretization, in which high-order Lagrange elements are used to approximate both the velocity and the embedding map. Meanwhile, the metric tensor is discretized using \emph{Regge elements}, a recently developed family of finite elements designed specifically for discretizing metric tensors on simplicial triangulations \cite{regge1961general,christiansen2004characterization, li2018regge}.
This choice of finite element pair—Lagrange elements for the velocity and embedding, and Regge elements for the metric tensor—naturally aligns within the framework of differential complexes. Specifically, in \cite{hu2022nonlinear}, isometric embedding is encoded in a nonlinear complex, which, upon linearization, corresponds precisely to the elasticity complex. Notably, Lagrange and Regge elements satisfy a discrete version of this elasticity complex, as established in \cite{christiansen2011linearization}.
Recent work by Gawlik et al. has initiated a systematic study of using finite elements to discretize intrinsic geometric quantities, such as the metric, connection, and intrinsic curvature \cite{berchenko2024finite, gawlik2020high, gopalakrishnan2023analysis, gopalakrishnan2023generalizing}, as well as to approximate intrinsic geometric flows \cite{gao2025ricci, gawlik2019finite}. However, the study of finite element methods for isometric embedding, as a fundamental problem, remains an open area of research. 

At the discrete level, under the hypothesis that the embedding map is well approximated, a discrete Korn inequality is established (see Lemma~\ref{lm:dis_Korn}), which characterizes the infinitesimal rigidity at the discrete level, as discussed in Remark \ref{rmk:rigid}, and further ensures the existence and uniqueness of the numerical solution.
The proof of this hypothesis constitutes the main theoretical result of this paper, i.e., we establish the convergence and error estimates of the finite element semi-discretization for the isometric embedding problem.
The primary challenge in the error analysis stems from the lack of full \(H^1\) coercivity. Specifically, coercivity is achieved through a surface Korn-type inequality (see \eqref{eq:korn}), where the tangential component is controlled in the \(H^1\)-norm, while the normal component is controlled only in the \(L^2\)-norm. As a result, coercivity degenerates in the normal direction.
However, in the stability estimate of the error equation, the nonlinear terms on the right-hand side, arising from perturbations, include the \(W^{1,\infty}\)- and \(H^1\)-norms of the error.
To address this issue, we employ inverse inequalities to reduce the \(W^{1,\infty}\)- and \(H^1\)-norms of the error to the \(L^2\)-norm (see \eqref{eq:inv}). This requires the use of high-order finite element spaces to ensure that these terms remain sufficiently small and can be controlled by the left-hand side via Gr\"onwall's inequality.
As a result, we establish the convergence and error estimates for the finite element semi-discretization of the isometric embedding problem for polynomial degrees \(k \ge 5\), as stated in Theorem~\ref{thm:err}.

The rest of this paper is organized as follows.
In Section~\ref{sec:isoemb}, we review the problem of isometric embedding of Riemannian metrics
and formulate the problems that will be addressed numerically.
In Section~\ref{sec:num}, we propose a new variational formulation for the isometric embedding problem,
and establish its well-posedness via a Korn inequality.
We then propose its finite element semi-discretization, followed by our main theorems on the well-posedness
and convergence of the numerical scheme.
A discrete Korn inequality is established in Section~\ref{sec:korn}.
It is used in Section~\ref{sec:pf_main}, which presents the proof of the main Theorem~\ref{thm:err}.
In Section~\ref{sec:numerical}, we present extensive numerical experiments to demonstrate the convergence
of the proposed method and to illustrate the simulation of isometric embeddings.

\section{Isometric embedding: background and notations}\label{sec:isoemb}

\subsection{Basic notations}

Let \(\Man\) be a closed, smooth, two-dimensional manifold embedded in \(\mathbb{R}^3\) via the inclusion \(i_{\Man}: \Man \hookrightarrow \mathbb{R}^3\); we assume that $\Man$ is diffeomorphic to the $2$-sphere $S^2$.
Let $\partial_i, i = 1,2$ be a local frame of the tangent bundle of $\Man$ and $\mathrm{d}x^i$ are the dual coframe with $\mathrm{d}x^i(\partial_j) = \delta_j^i$. Throughout the paper we adopt the Einstein summation convention over repeated indices.
We denote by $\mathrm{d}$ the exterior derivative on $\Man$. For a smooth function $f: \Man \rightarrow \mathbb{R}$, we write 
\[ \mathrm{d} f = \partial_i f \mathrm{d} x^i \in \Lambda^1(\Man),  \]
where $\Lambda^1(\Man)$ denotes the space of smooth differential 1-forms. 
We denote by $\odot$ the symmetrized tensor product, given maps $u,v: \Man \rightarrow \mathbb{R}^3$, the second order symmetric covariant tensor field $\mathrm{d} u \odot \mathrm{d} v$ is defined as
\[ \mathrm{d} u \odot \mathrm{d} v = \frac{1}{2}(\partial_i u \cdot \partial_j v 
+ \partial_i v \cdot \partial_j u ) \mathrm{d}x^i \otimes \mathrm{d} x^j \in S_2^0(\Man), \]
where $S_2^0(\Man)$ denotes the space of second order symmetric covariant tensor field on $\Man$.
Let $\varphi: \Man \rightarrow \mathcal{N}$ be a smooth map between the manifolds $\Man$ and $\mathcal{N}$. Its differential, denoted by $\varphi_*: T\Man \rightarrow T\mathcal{N}$, maps the tangent bundles 
$T\Man$ to $T\mathcal{N}$. This induces a mapping from the space of covariant tensor fields $S_2^0(\mathcal{N})$ to $S_2^0(\Man)$, called the \emph{pullback operator} and denoted by $\varphi^*$, defined by
\[
    (\varphi^* \sigma)(p)(\partial_i, \partial_j)
    = \sigma(\varphi(p))\bigl(\varphi_* \partial_i , \varphi_* \partial_j \bigr),
    \qquad \text{where } \sigma \in S_2^0(\mathcal{N}).
\]
The pullback operator commutes with both the exterior derivative and the tensor product \cite{johnlee}. 
Let $\delta_{\mathbb{R}^3}$ denote the Euclidean metric on $\mathbb{R}^3$.
The pullback of the Euclidean metric under the inclusion \(i_{\Man}\!: \Man \hookrightarrow \mathbb{R}^3\) induces a Riemannian metric \(g_{\Man}\) on \(\Man\), written
\[
g_{\Man} = (g_{\Man})_{ij}\,\mathrm{d}x^i \otimes \mathrm{d}x^j,\qquad
(g_{\Man})_{ij} = \partial_i i_{\Man}\cdot \partial_j i_{\Man},\ \ i,j=1,2 .
\]
This metric is referred to be the \emph{induced metric} on \(\Man\). 
We write \((g_{\Man})^{ij}\) for the components of the inverse matrix corresponding to \(((g_{\Man})_{ij})_{1 \leq i,j \leq 2}\), and \(\det(g_{\Man})\) for the determinant of this matrix.
The volume form induced by $g_\Man$ is then given by 
\[ \Vol_{\Man} = \sqrt{\det(g_{\Man})}\mathrm{d}x^1 \wedge \mathrm{d} x^2. \]
For tensor fields $\sigma,\omega$ and vector-valued functions $u,v:\Man\to\mathbb{R}^3$, we define the $L^2$ inner products on $\Man$ by
\begin{equation}\label{eq:inner-Man} 
(u,v)_\Man=\int_{\Man}(u\cdot v)\,\Vol_{\Man},\qquad
(\sigma,\omega)_\Man=\int_{\Man}\langle\sigma,\omega\rangle_{g_\Man}\,\Vol_{\Man}.
\end{equation} 
Here $(u\cdot v)$ denotes the Euclidean inner product in $\mathbb{R}^3$, and $\langle\sigma,\omega\rangle_{g_\Man}$ denotes the inner product induced by $g_\Man$, defined by
\[
\langle\sigma,\omega\rangle_{g_\Man}=(g_\Man)^{ik}(g_\Man)^{jl}\,\sigma_{ij}\,\omega_{kl},
\quad \text{with } \sigma_{ij}=\sigma(\partial_i,\partial_j)\ \text{and}\ \omega_{kl}=\omega(\partial_k,\partial_l).\]
We use $W^{k,p}(\Man)$ to denote the Sobolev space on $\Man$ with respect to the induced metric $g_{\Man}$, and denote $L^p(\Man) = W^{0,p}(\Man)$ and $H^k(\Man) = W^{k,2}(\Man)$. The space of $L^2$ second order symmetric covariant tensor field on $\Man$ is denoted by $L^2S_2^0(\Man)$. We refer to \cite{aubin1998some} for more detailed discussion of Sobolev space and corresponding theory on general Riemannian manifolds.

\subsection{Problems to be addressed numerically}

We are mainly concerned with \emph{Weyl's problem}~\cite{Weyl1916,Heinz1962}, which can be stated as follows. 
\begin{problem}[Weyl's problem]\label{Pro:iso-embd}
Given a smooth Riemmannian metric $g$ with positive Gaussian curvature on a two dimensional manifold $\Man$, find an embedding map $r:\Man\rightarrow \mathbb{R}^3$ such that
\begin{equation}\label{eq:isoemb-RM}
\mathrm{d}r \odot \mathrm{d}r = g.
\end{equation}
Here, \( \mathrm{d}r \odot \mathrm{d}r \) represents the pullback of the Euclidean metric via the embedding \( r \), i.e., 
\[
\mathrm{d}r \odot \mathrm{d}r = (\partial_i r \cdot \partial_j r)\,\mathrm{d}x^i \otimes \mathrm{d}x^j 
= (r^*\delta_{\mathbb{R}^3})(\partial_i, \partial_j)\,\mathrm{d}x^i \otimes \mathrm{d}x^j = r^*\delta_{\mathbb{R}^3} 
.
\]
\end{problem}

This problem is governed by a nonlinear and degenerate PDE system, and the corresponding weak formulation is far from straightforward.
Our approach for this problem is inspired by the continuity method developed in~\cite{Nirenberg1953,han2006isometric} for proving the existence of embeddings.
Let \( \Phi : S^2 \to \Man \) be a diffeomorphism. Then the pullback metric \( \Phi^* g \) on \( S^2 \) has positive Gaussian curvature: indeed, \( \Phi \) is an isometry between the Riemannian manifolds \( (S^2, \Phi^* g) \) and \( (\Man, g) \), and hence it preserves Gaussian curvature.
By the uniformization theorem for two-dimensional Riemannian manifolds (see, for instance,~\cite[Lemma~9.1.4]{han2006isometric}), there exists a continuous family of metrics
\(
\{ g_{S^2}(t) \}_{0 \le t \le 1}
\)
on \( S^2 \) such that each \( g_{S^2}(t) \) has positive Gaussian curvature and
\[
g_{S^2}(1) = \Phi^* g,
\qquad
g_{S^2}(0) = i_{S^2}^* \delta_{\mathbb{R}^3},
\]
where \( i_{S^2} : S^2 \hookrightarrow \mathbb{R}^3 \) denotes the standard embedding.
This family of metrics on \( S^2 \) induces a corresponding family of time-dependent metrics on \( \Man \) via
\[
g(t) = (\Phi^{-1})^* g_{S^2}(t),
\]
which provides a continuous path connecting the initial metric
\(
g(0) = (i_{S^2} \circ \Phi^{-1})^* \delta_{\mathbb{R}^3}
\)
to the target metric \( g(1) = g \). By the same reasoning as above, each metric \( g(t) \) along this path has positive Gaussian curvature.

In this way, we transform the static problem~\eqref{eq:isoemb-RM} into a dynamic one.
Starting from \( r(0)= i_{S^2}\circ \Phi^{-1} \), we seek an \emph{embedding flow}
\( r(t): \Man \to \mathbb{R}^3 \) that realizes the evolving metric \( g(t) \) for each \( t \in [0,1] \), with \( r(1) \) being a solution to~\eqref{eq:isoemb-RM}.
This dynamic formulation is also natural in the geometric visualization of intrinsic curvature flows.
For intrinsic flows such as the Ricci flow, the unknown is a time-dependent Riemannian metric \( g(t) \);
to visualize its evolution, one seeks a family of isometric embeddings \( r(t): \Man \to \mathbb{R}^3 \) realizing \( g(t) \) at each time.
The evolving geometry can then be tracked through the embedded surfaces
\[
\Gamma(t)=\{\, r(p,t): p\in\Man \,\}\subset \mathbb{R}^3 .
\]
Motivated by these two scenarios (finding an embedding flow for solving \eqref{eq:isoemb-RM} and isometrically embedding intrinsic curvature flows into Euclidean space), we formulate our second problem as follows.
\begin{problem}\label{Pro:dynamic}
Let \( g(t) \) be a time-dependent Riemannian metric on \( \Man \) whose Gaussian curvature remains strictly positive for all \( t\in[0,T] \).
Given an initial condition \( r(0) \), find an embedding flow
\[
r(t): \Man\times[0,T]\to\mathbb{R}^3
\]
such that
\begin{equation}\label{eq:isoemb}
\mathrm{d}r(t)\odot \mathrm{d}r(t)=g(t),\qquad t\in[0,T].
\end{equation}
\end{problem}

The existence of solutions to Problems~\ref{Pro:iso-embd} and~\ref{Pro:dynamic} is guaranteed by Weyl's embedding theorem; see, e.g.,~\cite[Theorem~9.0.1]{han2006isometric}.
It states that any metric with strictly positive Gaussian curvature on a closed two-dimensional manifold admits an isometric embedding into \( \mathbb{R}^3 \).
The uniqueness of such embeddings, referred to as \emph{rigidity}, is discussed in~\cite[Chapter~8]{han2006isometric}: under the same assumption on the metric, any two isometric embeddings \( r_1 \) and \( r_2 \) differ only by a rigid motion of \( \mathbb{R}^3 \), i.e., there exist \( Q \in \mathrm{O}(3) \) and \( \beta \in \mathbb{R}^3 \) such that
\begin{equation}\label{eq:Rigidity} 
r_1 = Q r_2 + \beta .
\end{equation}

Both the isometric embedding of a fixed metric and the visualization of intrinsic curvature flows can be reduced to solving a sequence of dynamic embedding problems of the form in Problem~\ref{Pro:dynamic}.
However, the embedding flow \( r(t) \) in Problem~\ref{Pro:dynamic} still satisfies a nonlinear PDE system, and the solution is unique only up to rigid motions.
Crucially, the \emph{dynamic} nature of Problem~\ref{Pro:dynamic} allows us to differentiate~\eqref{eq:isoemb} in time and instead study the PDE governing the velocity:
\begin{equation}\label{eq:pt_isoemb}
2\,\mathrm{d}r \odot \mathrm{d}v = \partial_t g,
\quad \text{ where } \quad 
\partial_t r = v,
\end{equation}
which is a linear PDE for \(v\) and can be viewed as the linearization of~\eqref{eq:isoemb}.
The existence of solutions to~\eqref{eq:pt_isoemb} is established in~\cite[Lemma~9.2.2]{han2006isometric} under the same curvature assumption, namely that \( g(t) \) has strictly positive Gaussian curvature.
The corresponding uniqueness, referred to as \emph{infinitesimal rigidity}, is discussed in~\cite[Theorem~8.2.3]{han2006isometric}:
under the same assumption on \( g(t) \), any two solutions \( v_1 \) and \( v_2 \) of~\eqref{eq:pt_isoemb} satisfy
\begin{equation}\label{eq:inf-rigid}
v_1 - v_2 \in \RM[r],
\qquad
\text{where}\quad
\RM[r] = \{\, \alpha \times r + \beta \mid \alpha, \beta \in \mathbb{R}^3 \,\}.
\end{equation}
Here \( \RM[r] \) is the space of \emph{infinitesimal rigid motions}, which is a six-dimensional linear space.

Motivated by the discussion above, we summarize the existence and uniqueness of a smooth embedding flow \(r(t)\) satisfying~\eqref{eq:isoemb}, whose time derivative \(\partial_t r\) solves~\eqref{eq:pt_isoemb} and is uniquely determined by the orthogonality condition~\eqref{eq:ptr_orth_RM}. The proof is referred to \ref{appendix:Wely-emb}.
\begin{theorem}\label{thm:wely-emb}
Let \(g(t)\) be a smooth Riemannian metric on \(\Man\), depending smoothly on \(t\in[0,T]\), with smooth time derivative \(\partial_t g(t)\).
Assume that the Gaussian curvature of \(g(t)\) remains strictly positive for all \(t\in[0,T]\).
Then there exists a {unique} smooth embedding flow \(r(t):\Man\to\mathbb{R}^3\), depending smoothly on \(t\in[0,T]\), such that \(r(t)\) satisfies~\eqref{eq:isoemb} and its time derivative \(\partial_t r\) satisfies~\eqref{eq:pt_isoemb} together with the orthogonality condition
\begin{equation}\label{eq:ptr_orth_RM}
\int_{\Man} \partial_t r \cdot (\alpha \times r + \beta)\,\Vol_{\Man} = 0,
\qquad \forall\, \alpha,\beta \in \mathbb{R}^3.
\end{equation}
\end{theorem}

\begin{remark}[Positive Gaussian curvature]
\label{rmk:posi_curv}
\upshape
The assumption of strictly positive Gaussian curvature plays a crucial role in establishing existence and uniqueness for Problems~\ref{Pro:iso-embd} and~\ref{Pro:dynamic}, as well as for the linearized equation~\eqref{eq:pt_isoemb}.
In particular, this assumption is satisfied in the following two typical situations.\medskip

\noindent\textit{(1) Conformal deformation.}
By the uniformization theorem, any metric \(g\) on \(S^2\) with positive Gaussian curvature is conformal to the standard metric \(g_0\coloneqq i_{S^2}^*\delta_{\mathbb{R}^3}\), i.e., there exists \(\lambda:S^2\to\mathbb{R}_+\) such that
\(
g = e^{2\lambda} g_0 .
\)
Define a conformal path \(g(t)\) for \(t\in[0,1]\) by
\[
g(t) = e^{2t\lambda}\, g_0 .
\]
Then the Gaussian curvature of \(g(t)\) remains strictly positive for all \(t\in[0,1]\).
Indeed, using the standard transformation law of Gaussian curvature under conformal changes, the Gaussian curvature \(\kappa(t)\) of \(g(t)\) satisfies
\[
\kappa(t) = e^{-2t\lambda}\bigl(\kappa_0 - t\,\Delta_{g_0}\lambda\bigr),
\]
where \(\kappa_0>0\) is the Gaussian curvature of \(g_0 \).
Since \(g(1)=e^{2\lambda}g_0\), we also have
\[
\kappa_1 = e^{-2\lambda}\bigl(\kappa_0-\Delta_{g_0}\lambda\bigr),
\]
and hence
\[
\kappa(t)
= e^{-2t\lambda}\bigl((1-t)\kappa_0 + t\,e^{2\lambda}\kappa_1\bigr) > 0,
\]
because \(\kappa_0>0\) and \(\kappa_1>0\).\medskip

\noindent\textit{(2) Intrinsic geometric flows.}
Another typical situation arises when \(g(t)\) evolves according to an intrinsic curvature flow.
For example, under the Ricci flow on a closed surface,
\[
\partial_t g(t) = -2\,\kappa(t)\, g(t),
\]
one can show that the Gaussian curvature \(\kappa(t)\) satisfies the parabolic evolution equation
\[
\partial_t \kappa(t) = \Delta_{g(t)} \kappa(t) + 2\,\kappa(t)^2,
\]
and hence remains positive for all time provided \(\kappa(0)>0\); see~\cite[Corollary~2.11]{chow2023hamilton}.
\end{remark}

Throughout this paper, we denote by $C$ and $h_{0}$ two generic positive constants which are different at different occurrences, possibly depending on the exact solution and time $T$, but are independent of the mesh size $h$ and $t\in[0,T]$. The notation $X \lesssim Y$ means $X \leq C Y$ for some constant $C$, and $X \eqsim Y$ means $X \lesssim Y$ and $Y \lesssim X$.

\section{Weak formulation and numerical scheme}\label{sec:num}

\subsection{New weak formulation}

In what follows, we propose a new variational formulation to determine the velocity \(v=\partial_t r\) satisfying~\eqref{eq:pt_isoemb} and~\eqref{eq:ptr_orth_RM}, and then use it to recover the embedding flow \(r(t)\) obtained in Theorem~\ref{thm:wely-emb}.
This variational formulation is well suited for the finite element discretization developed in the subsequent sections.

We first introduce the following notation. On the evolving surface
\(
\Gamma(t)=\{\,r(p,t):p\in\Man\,\},
\)
the unit normal and the orthogonal projection onto the tangent plane are given by
\[
n=\frac{\partial_1 r\times \partial_2 r}{\lvert \partial_1 r\times \partial_2 r\rvert},
\qquad
P=I-nn^{\top}.
\]
For a smooth embedding \(r:\Man\to\mathbb{R}^3\), we view the left-hand side of~\eqref{eq:pt_isoemb} as a linear operator with respect to \(v\), denoted by \({\rm D}_r\):
for any vector field \(v:\Man\to\mathbb{R}^3\), set
\begin{equation}\label{eq:def-Dr}
{\rm D}_r v \coloneqq \mathrm{d}r \odot \mathrm{d}v
= \tfrac12\big(\partial_i r\cdot \partial_j v+\partial_j r\cdot \partial_i v\big)\,
\mathrm{d}x^i\otimes \mathrm{d}x^j .
\end{equation}
Decomposing \(v\) into tangential and normal components yields
\begin{equation}\label{eq:TanH1_NorL2}
\begin{aligned}
{\rm D}_r v
&= {\rm D}_r(Pv) + \mathrm{d}r \odot \mathrm{d}\big((n\cdot v)n\big) \\
&= {\rm D}_r(Pv)
+ \tfrac12\Big(\partial_i r\cdot \partial_j\big((n\cdot v)n\big)
            +\partial_j r\cdot \partial_i\big((n\cdot v)n\big)\Big)\,
   \mathrm{d}x^i\otimes \mathrm{d}x^j \\
&= {\rm D}_r(Pv) - (\partial_{ij}r\cdot n)\,(n\cdot v)\,\mathrm{d}x^i\otimes \mathrm{d}x^j .
\end{aligned}
\end{equation}
Hence, \({\rm D}_r v\) is well defined in \(L^2\) whenever the normal component \(n\cdot v\in L^2\) and the tangential component \(Pv\in H^1\).
This motivates the Hilbert space
\[
H^1_T(\Man;\mathbb{R}^3)
\coloneqq\{\, v\in L^2(\Man;\mathbb{R}^3)\ \mid\ Pv\in H^1(\Man;\mathbb{R}^3)\,\}.
\]
Accordingly, \({\rm D}_r\) can be regarded as a bounded linear operator
\(
{\rm D}_r: H^1_T(\Man;\mathbb{R}^3)\to L^2S^0_2(\Man).
\)
Let
\[
\Ker({\rm D}_r)\coloneqq\{\, v\in H^1_T(\Man;\mathbb{R}^3)\ \mid\ {\rm D}_r v=0\,\}.
\]
Then the infinitesimal rigidity defined in ~\eqref{eq:inf-rigid} can be equivalently stated as
\begin{equation}\label{eq:rigid}
\RM[r]=\Ker({\rm D}_r).
\end{equation}

We are now ready to recast~\eqref{eq:isoemb}, \eqref{eq:pt_isoemb}, and~\eqref{eq:ptr_orth_RM} into a variational formulation.
Find \(r(t),v(t):\Man\to\mathbb{R}^3\) and \( \lambda(t)\in \RM[r(t)] \) such that
\begin{subequations}\label{eq:vel}
\begin{align}
\label{eq:vel-r}
\partial_t r &= v, \\
\label{eq:vel-a}
2\,({\rm D}_r v, {\rm D}_r q)_{\Man} + (\lambda, q)_{\Man}
&= (\partial_t g, {\rm D}_r q)_{\Man}, \\
\label{eq:vel-b}
(v,\mu)_{\Man} &= 0,
\end{align}
\end{subequations}
for all test functions \(q\in H^1_T(\Man;\mathbb{R}^3)\) and \(\mu\in \RM[r(t)]\).
It is straightforward to verify that the triple \((r,\partial_t r,0)\) satisfies~\eqref{eq:vel}, where \(r\) is the embedding flow provided by Theorem~\ref{thm:wely-emb}.

Equations~\eqref{eq:vel-a}-\eqref{eq:vel-b} form a linear system for the velocity \(v \in H^1_T(\Man;\mathbb{R}^3)\) and the multiplier \(\lambda \in \RM[r]\), where \( \RM[r] \) is a finite-dimensional linear space (\(\dim \RM[r] = 6\)), and all norms on this space are equivalent. Therefore, based on the abstract theory in \cite{boffi2013mixed}, the well-posedness of this linear system hinges on a Korn inequality on the manifold, as stated in Theorem~\ref{lm:Korn}.
Since we were unaware of a Korn inequality of this form in the existing literature, we provide a proof of it in~\ref{appendix:Korn}. The proof builds upon the study of the linearized equation for isometric embeddings in \cite[Lemma 9.2.2]{han2006isometric}.
\begin{theorem}\label{lm:Korn} 
    Let \(r:\Man \rightarrow \mathbb{R}^3\) be a smooth embedding of \(g\) with positive Gaussian curvature. Then
\begin{equation}\label{eq:korn} 
\|v\|_{L^2(\Man)} + \| P v \|_{H^1(\Man)} \lesssim \| {\rm D}_r v \|_{L^2(\Man)},
\quad \forall v \in (\RM[r])^{\perp},
\end{equation} 
where 
\[
(\RM[r])^{\perp} 
= \{ w \in H^1_{T}(\Man; \mathbb{R}^3) \mid (w, \mu)_{\Man} = 0, \forall \mu \in \RM[r] \}.
\]
\end{theorem}

We are now ready to establish the well-posedness of linear system \eqref{eq:vel-a} and~\eqref{eq:vel-b}.
\begin{theorem}\label{thm:wellpose}
Let \(r:\Man \rightarrow \mathbb{R}^3\) be a smooth embedding of $g$ with positive Gaussian curvature. Then, the system to find $v \in H^1_{T}(\Man; \mathbb{R}^3)$ and $\lambda \in \RM[r]$ such that for any  $q \in  H^1_{T}(\Man; \mathbb{R}^3)$ and $\mu \in \RM[r]$
\begin{subequations}\label{Mix:system}
\begin{align}
2\,({\rm D}_r v, {\rm D}_r q)_{\Man}
+ (\lambda, q)_{\Man} &= (\partial_t g, {\rm D}_r q)_{\Man}, \\
(v, \mu)_{\Man} &= 0,
\end{align}
\end{subequations}
admits a unique solution.
\end{theorem}

\begin{proof}
We prove the well-posedness of system using the abstract theory in \cite{boffi2013mixed}.
To this end, we define the bilinear forms:
\begin{align*}
a(u, q) &= 2({\rm D}_ru, {\rm D}_rq)_{\Man}, 
\quad && \forall u, q \in H^1_T(\Man; \mathbb{R}^3)
\\
b(q, \lambda) &= (q,\lambda)_{\Man}, \quad && \forall q \in H^1_T(\Man; \mathbb{R}^3),
\text{ and } \lambda \in \RM[r]. 
\end{align*}
It follows from Korn's inequality~\eqref{eq:korn} that
\[
a(q, q) = 2 \| {\rm D}_rq \|_{L^2(\Man)}^2 \gtrsim \| q \|_{L^2(\Man)}^2 
+ \|P q\|_{H^1(\Man)}^2
, \quad \forall q \in (\RM[r])^\perp.
\]
Next, for any \( \lambda \in \RM[r] \), we take 
\( q = \lambda \in H^1_T(\Man; \mathbb{R}^3) \). 
By the equivalence of norms in finite-dimensional spaces 
(noting that \( \dim \RM[r] = 6 \)), we obtain
\[
    b(q, \lambda) = \|\lambda\|_{L^2(\Man)}^2 
    \quad \text{and} \quad
    \|q\|_{L^2(\Man)} + \|Pq\|_{H^1(\Man)}
    \eqsim \|q\|_{L^2(\Man)}
    = \|\lambda\|_{L^2(\Man)} .
\]
In particular, this shows that
\[
    \inf_{\lambda \in \RM[r]} 
    \sup_{q \in H^1_{T}(\Man; \mathbb{R}^3)}
    \frac{
        b(q, \lambda)
    }{
        \|\lambda\|_{L^2(\Man)}
        \bigl(\|q\|_{L^2(\Man)} + \|Pq\|_{H^1(\Man)}\bigr)
    }
    \gtrsim 1,
\]
that is, the inf-sup stability of \( b(\cdot, \cdot) \) is ensured.

Therefore, by the abstract theory in \cite{boffi2013mixed}, the saddle-point system \eqref{Mix:system} admits a unique solution \((v,\lambda)\) and satisfies the stability estimate
\[
 \|v\|_{L^2(\Man)} + \|P v\|_{H^1(\Man)} + \|\lambda\|_{L^2(\Man)} \,\lesssim\, \|\partial_t g \|_{L^2(\Man)},
\]
where the hidden constant depends continuously on the smooth embedding \(r\).
\end{proof}

\smallskip

\begin{remark} [Evolution equation for isometric embedding]
    \upshape
    Let \({\rm F}(t, r, \mathrm{d}r)\) denote the velocity uniquely determined by the linear system \eqref{eq:vel-a}–\eqref{eq:vel-b}. Then the formulation \eqref{eq:vel} can be recast as the evolution equation: Find \(r:[0,T]\times\Man\to\mathbb{R}^3\) such that
  \[\partial_t r(t) = {\rm F}(t, r(t), \mathrm{d}r(t)).\] 
\end{remark}

\begin{remark} [Graph norm on $H^1_T(\Man; \mathbb{R}^3)$]
    \upshape 
    We define a graph norm associated with the operator \({\rm D}_r\) in \eqref{eq:def-Dr}:
    \begin{equation}\label{eq:graph_norm}
        \normt{v}_{\Man}^2
        \coloneqq \|v\|_{L^2(\Man)}^2 + \| {\rm D}_r v \|_{L^2(\Man)}^2,
        \qquad \forall\, v \in H^1_{T}(\Man; \mathbb{R}^3).
    \end{equation}
    We decompose \(v \in H^1_{T}(\Man; \mathbb{R}^3)\) orthogonally as \(v_{\RM} \in \RM[r]\) and \(v_{\perp} \in (\RM[r])^{\perp}\); then the Korn inequality \eqref{eq:korn}, the norm equivalence in the finite-dimensional space \(\RM[r]\), and \eqref{eq:rigid} lead to 
    \[
    \begin{aligned}
        \| P v_{\perp} \|_{H^1(\Man)} + \| v_{\perp} \|_{L^2(\Man)}
        &\lesssim \| {\rm D}_r v_{\perp} \|_{L^2(\Man)} =  \| {\rm D}_r v \|_{L^2(\Man)}, \\ 
        \| P v_{\RM} \|_{H^1(\Man)} + \| v_{\RM} \|_{L^2(\Man)} 
        &\eqsim \| v_{\RM} \|_{L^2(\Man)} \le \| v \|_{L^2(\Man)}. 
    \end{aligned}
    \]
    This implies that \( \| v \|_{L^2(\Man)} + \|Pv\|_{H^1(\Man)} \lesssim  \normt{ v }_{\Man} \).
    Conversely, by the identity in \eqref{eq:TanH1_NorL2} we obtain
    \[
    \normt{ v }_{\Man} \lesssim \| v \|_{L^2(\Man)} + \|Pv\|_{H^1(\Man)}.
    \]
  Therefore, the graph norm in \eqref{eq:graph_norm} defines an equivalent norm on the Hilbert space \(H^1_{T}(\Man; \mathbb{R}^3)\).
In the following numerical analysis, we will use the graph norm in \eqref{eq:graph_norm} and its discrete version, defined in \eqref{eq:dis_graph_norm}, for error analysis.    
\end{remark}

\subsection{Finite element method}

For the purpose of discretization, we regard the manifold $\Man$ as an closed and oriented Euclidean hypersurface in $\mathbb{R}^3$, which allows us to define the \emph{signed distance function} $\operatorname{dist}(\cdot,\Man): \mathbb{R}^3 \rightarrow \mathbb{R}$.  This function is defined such that $|\operatorname{dist}(p,\Man)|$ represents the Euclidean distance from $p$ to $\Man$, and $\operatorname{dist}(\cdot,\Man)$ is negative inside the bounded region enclosed by $\Man$ and positive outside in the unbounded region.
Let \(\Man_h\) be a quasi-uniform, piecewise-flat triangulated surface with vertices on \(\Man\) and mesh size \(h\), assuming that
\(\Man_h \subset D_{\epsilon}(\Man)=\{\,p\in\mathbb R^{3}\mid |\operatorname{dist}(p,\Man)|\le \epsilon\,\}\) for some small \(\epsilon>0\).
Then there exists a bijection \(a:\Man_h\to\Man\), called the \emph{closest point projection}, determined by
\[
a(p) = p - \operatorname{dist}(p,\Man)\, n_{\Man}(a(p)),
\]
where \(n_{\Man}:\Man\to\mathbb R^{3}\) denotes the outer unit normal vector on \(\Man\).
 Via the closest point projection, any function \(f_h: \Man_h \to \mathbb{R}\) can be lifted to the smooth surface \(\Man\), and, conversely, any function \(f: \Man \to \mathbb{R}\) can be lifted back to the discrete surface \(\Man_h\), namely
\[
f_h^{(\ell)} = f_h \circ a^{-1}: \Man \to \mathbb{R},
\qquad
f^{(-\ell)} = f \circ a: \Man_h \to \mathbb{R}.
\]
Similarly, for a tensor field \(\sigma\) on \(\Man\) and a (piecewise) tensor field \(\sigma_h\) on \(\Man_h\), we set
\[
\sigma_h^{(\ell)} = (a^{-1})^* \sigma_h,
\qquad
\sigma^{(-\ell)} = a^* \sigma,
\]
where \(a^*\) and \((a^{-1})^*\) denote the pullback operators.

We denote by $\mathcal{T}_h$ the set of triangles in $\Man_h$ and by $\mathcal{E}_h$ the set of its edges.
Let \(V_h\) be the Lagrange finite element space of polynomial degree \(k\ge 1\) on \(\Man_h\):
\[
V_h = \{\, v \in C(\Man_h) \mid v|_K \in P_k(K) \text{ for all } K \in \mathcal{T}_h  \,\}.
\]
The finite element space extends naturally to the vector-valued space \(V_h^3\), consisting of three components. For a given nonzero function \(r_h\in V_h^3\), we define the associated infinitesimal rigid motion space by
\[
\RM[r_h] = \{\, \alpha \times r_h + \beta \mid \alpha, \beta \in \mathbb{R}^3 \,\} \subset V_h^3,
\qquad \dim \RM[r_h] = 6.
\]

On a triangle \(K\subset\Man_h\), let \(S_2^0(K)\) denote the space of all symmetric \((0,2)\)-tensor fields on \(K\).
We say that \(\sigma\in H^1 S_2^0(K)\) if each component of its coordinate representation belongs to \(H^1(K)\).
We define \(\Sigma\) by
\[
\Sigma = \Big\{\, \sigma \in \prod_{K\in\mathcal T_h} H^1 S^0_2(K) \ \Big|\ 
i_{K_+,e}^*(\sigma|_{K_+}) = i_{K_-,e}^*(\sigma|_{K_-})
\ \text{for all } e=K_+\cap K_- \in \mathcal E_h \,\Big\},
\]
where \(i_{K_{\pm},e}: e \hookrightarrow K_{\pm}\) denotes the inclusion and \(i_{K_{\pm},e}^*\) its pullback.
Thus, \(\Sigma\) is the space of piecewise symmetric \((0,2)\)-tensor fields on \(\Man_h\) whose tangential--tangential components are continuous across edges.
For example, since the closest point projection \(a:\Man_h\to\Man\) is continuous, we have
\(a\circ i_{K_+,e}=a\circ i_{K_-,e}\) for all edges \(e=K_+\cap K_-\in\mathcal E_h\).
Consequently, any smooth tensor field \(\sigma\in S_2^0(\Man)\) admits the pullback \(\sigma^{(-\ell)}=a^*\sigma\), which belongs to \(\Sigma\) because
\[
i_{K_+,e}^* \sigma^{(-\ell)} = (a \circ i_{K_+,e})^* \sigma
= (a \circ i_{K_-,e})^* \sigma = i_{K_-,e}^* \sigma^{(-\ell)}.
\]
In particular, the pullback of the target metric satisfies \(g^{(-\ell)}(t)=a^*g(t)\in\Sigma\), and its time derivative satisfies
\(\partial_t g^{(-\ell)}(t)=a^*(\partial_t g(t))\in\Sigma\). 
To obtain a high-order discretization of the metric tensor, we employ the Regge finite element space \cite{christiansen2004characterization,li2018regge} of degree \(k_g \ge 0\), defined by
\[
\Sigma_h = \big\{\, \sigma \in \Sigma \ \big|\ \sigma|_K \in P_{k_g} S_2^0(K) \text{ for all } K \in \mathcal{T}_h \,\big\},
\]
where \(P_{k_g}S_2^0(K)\) denotes the space of symmetric \((0,2)\)-tensor fields on \(K\) whose components are polynomials of degree at most \(k_g\).
We refer to \cite[Chapter~2]{li2018regge} for details on the shape function space, degrees of freedom, and unisolvence of $\Sigma_h$.  
Regge finite elements provide a high-order approximation of smooth tensor fields \(\sigma\in S_2^0(\Man)\).
Given \(\sigma\), we first pull it back to \(\Man_h\) via the closest point projection, i.e.,\(\sigma^{(-\ell)} = a^*\sigma \in \Sigma\).
We then evaluate the degrees of freedom of \(\Sigma_h\) on \(\sigma^{(-\ell)}\) to obtain the Regge approximation \(R_h\sigma^{(-\ell)} \in \Sigma_h\).
The operator \(R_h:\Sigma\to\Sigma_h\) is the \emph{canonical Regge interpolation} introduced in \cite[Section~2.3.2]{li2018regge}.
It satisfies a quasi-optimal approximation estimate for (piecewise) smooth tensor fields; see \cite[Theorem~2.5]{li2018regge}.
In what follows, we approximate the two metric tensors appearing in the variational formulation \eqref{eq:vel} using Regge finite elements.

\textit{(1) The induced metric \(g_{\Man}\)}:
The induced metric \(g_{\Man}\) is approximated by its Regge interpolant, defined by
\(g_{\Man_h} \coloneqq R_h(a^* g_{\Man})\).
By the approximation property of \(R_h\), for \(h\) sufficiently small,
\(g_{\Man_h}\in\Sigma_h\) defines a discrete Riemannian metric on \(\Man_h\). 
Accordingly, we write \(\Vol_{g_{\Man_h}}\) for the volume form on \(\Man_h\) associated with \(g_{\Man_h}\). For tensor fields \(\sigma,\omega\) on \(\Man_h\) and vector fields \(u,v: \Man_h\to\mathbb R^3\), we set
\begin{equation}\label{eq:inner_Mh} 
(u,v)_{\Man_h} = \int_{\Man_h} (u\cdot v)\,\Vol_{g_{\Man_h}},
\qquad
(\sigma,\omega)_{\Man_h} = \int_{\Man_h} \langle \sigma,\omega\rangle_{g_{\Man_h}}\,\Vol_{g_{\Man_h}},
\end{equation} 
where \(\langle \sigma,\omega\rangle_{g_{\Man_h}}\) is the inner product induced by \(g_{\Man_h}\).
We denote by \(L^p(\Man_h)\), \(W^{1,p}(\Man_h)\) \((1\le p\le\infty)\), and \(H^1(\Man_h)\) the Sobolev spaces on \(\Man_h\) with respect to the metric \(g_{\Man_h}\).
The approximation property of \(R_h\) in \cite[Theorem~2.5]{li2018regge} yields
\begin{equation}\label{eq:gMh_app}
\|g_{\Man_h}-a^*g_{\Man}\|_{L^p(\Man_h)} = 
\|R_h(a^*g_{\Man})-a^*g_{\Man}\|_{L^p(\Man_h)}
\lesssim h^{k_g+1}, \qquad \forall\, 1\le p\le\infty .
\end{equation}
This implies a high-order geometric consistency estimate between 
$(\Man,g_{\Man})$ and $(\Man_h,g_{\Man_h})$:
\begin{equation}\label{eq:geo_err_vector}
\big|(u,v)_{\Man_h} - (u^{(\ell)},v^{(\ell)})_{\Man}\big|
\lesssim h^{k_g+1}\,\|u\|_{L^2(\Man_h)}\,\|v\|_{L^2(\Man_h)}.
\end{equation}
Indeed, denote by $\Vol_{a^*g_{\Man}}$ the volume form on $\Man_h$ induced by the 
pullback metric $a^*g_{\Man}$. Then we have
\[
\begin{aligned}
    (u,v)_{\Man_h} - (u^{(\ell)},v^{(\ell)})_{\Man}
    &= \int_{\Man_h} u v\,\Vol_{g_{\Man_h}}
     - \int_{\Man_h} u v\,\Vol_{a^*g_{\Man}} \\
    &\lesssim \|u\|_{L^2(\Man_h)}\,\|v\|_{L^2(\Man_h)}\,
        \|\Vol_{g_{\Man_h}} - \Vol_{a^*g_{\Man}}\|_{L^{\infty}(\Man_h)} \\
    &\lesssim \|u\|_{L^2(\Man_h)}\,\|v\|_{L^2(\Man_h)}\,
        \|g_{\Man_h} - a^*g_{\Man}\|_{L^{\infty}(\Man_h)},
\end{aligned}
\]
which, combined with~\eqref{eq:gMh_app}, proves~\eqref{eq:geo_err_vector}.
By an analogous argument, we obtain the estimates 
\begin{subequations}\label{eq:geo_err_du_tensor}
\begin{align}
\label{eq:geo_err_du} 
\bigl|(\mathrm{d}u,\mathrm{d}v)_{\Man_h} 
 - (\mathrm{d}u^{(\ell)},\mathrm{d}v^{(\ell)})_{\Man}\bigr|
&\lesssim h^{k_g+1}\,\|\mathrm{d}u\|_{L^2(\Man_h)}\,\|\mathrm{d}v\|_{L^2(\Man_h)}, \\
\label{eq:geo_err_tensor}
\bigl|(\sigma,\omega)_{\Man_h} 
 - (\sigma^{(\ell)},\omega^{(\ell)})_{\Man}\bigr|
&\lesssim h^{k_g+1}\,\|\sigma\|_{L^2(\Man_h)}\,\|\omega\|_{L^2(\Man_h)}.
\end{align}
\end{subequations}
Hence, for sufficiently small $h$, by combining \eqref{eq:geo_err_du_tensor} 
and \eqref{eq:geo_err_vector} with $u=v$ and $\omega=\sigma$, 
we derive the following norm equivalences:
\begin{equation}\label{eq:norm-equi} 
\|v\|_{H^{l}(\Man_h)} 
\eqsim 
\|v^{(\ell)}\|_{H^{l}(\Man)}, 
\qquad 
\|\sigma\|_{L^2(\Man_h)}
\eqsim 
\|\sigma^{(\ell)}\|_{L^2(\Man)}, 
\qquad 
l \in \{0,1\},
\end{equation}
for all $v \in H^{l}(\Man_h)$ and $\sigma \in L^2 S_{2}^{0}(\Man_h)$.

\textit{(2) The target metric \(g(t)\)}:
We also approximate the time-dependent metric \(g(t)\) that appears on the right-hand side of \eqref{eq:isoemb}.
Its pullback \(g^{(-\ell)}(t)\coloneqq a^*g(t)\) satisfies the identities
\begin{equation}\label{eq:g_ell}
g^{(-\ell)} = a^*(\mathrm{d}r\odot \mathrm{d}r)
= \mathrm{d}r^{(-\ell)}\odot \mathrm{d}r^{(-\ell)} \in \Sigma,
\quad\text{and}\quad
\partial_t \, g^{(-\ell)} = 2\,\mathrm{d}r^{(-\ell)}\odot \mathrm{d}\partial_t r^{(-\ell)} \in \Sigma,
\end{equation}
where \(r\) is the smooth solution of \eqref{eq:isoemb}. Here we use that the pullback commutes with the exterior derivative and with the symmetric tensor product \cite{johnlee}.
Let \(g_h(t)\in\Sigma_h\) be a Regge finite element approximation of the target metric \(g(t)\), and assume that, uniformly for \(t\in[0,T]\), it satisfies 
\begin{equation}\label{eq:gh_app} 
    \|g_h(t) - g^{(-\ell)}(t)\|_{L^2(\Man_h)} \lesssim h^k, \qquad \|\partial_t g_h(t) - \partial_t g^{(-\ell)}(t)\|_{L^2(\Man_h)} \lesssim h^k. 
\end{equation}
One admissible choice is to take \(g_h(t)\) as the Regge interpolant of \(g^{(-\ell)}(t)\), namely \(g_h(t)=R_h g^{(-\ell)}(t)\in\Sigma_h\), in which case the approximation property in Theorem~2.5 of \cite{li2018regge} implies that \eqref{eq:gh_app} holds provided that \(k_g+1\ge k\) and \(g(t)\) is sufficiently smooth.
Alternatively, \(g_h(t)\) can be chosen as the numerical solution produced by a finite element discretization of an intrinsic curvature flow, for example, by the Ricci flow scheme proposed in \cite{gawlik2019finite, gao2025ricci}. In this case, the corresponding error estimate (see Theorem~3.1 in \cite{gao2025ricci}) yields \eqref{eq:gh_app}.

Hence, two Regge finite element approximations are employed in the numerical scheme.
The first is \(g_{\Man_h}\), a fixed reference metric on \(\Man_h\); all finite element integrals are taken with respect to \(g_{\Man_h}\).
The second is an evolving Regge metric \(g_h(t)\), which approximates the traget metric \(g(t)\) on the right-hand side of \eqref{eq:isoemb} and satisfies \eqref{eq:gh_app}. 
We are now ready to propose the semi-discrete finite element discretization of \eqref{eq:vel}: find \(r_h(t), v_h(t)\in V_h^3\) and \(\lambda_h(t)\in \RM[r_h(t)]\) satisfying 
\begin{subequations}\label{eq:num}  
\begin{align}
    \hspace{-28mm}
\label{eq:num-v}    
\partial_t r_h &= v_h, \\    
\label{eq:num-a} 
2(\mathrm{d} r_h \odot \mathrm{d} v_h, \mathrm{d} r_h \odot \mathrm{d} q_h)_{\Man_h} 
+ (\lambda_h, q_h)_{\Man_h}  &= 
(\partial_t g_h, \mathrm{d} r_h \odot \mathrm{d} q_h)_{\Man_h}, \\ 
\label{eq:num-b} 
(v_h, \mu_h)_{\Man_h} &= 0,
\end{align}
\end{subequations}
for all test functions $q_h \in V_h^3$ and $\mu_h \in \RM[r_h(t)]$.

\subsection{Well-posedness and error estimates}

The main theoretical results of this paper are the wellposedness and the convergence for the scheme \eqref{eq:num}, which is summarized in the following theorem.

\begin{theorem}\label{thm:err}
Under the assumptions of Theorem~\ref{thm:wely-emb}, let \(r(t)\) be the smooth solution obtained therein.
Assume that the polynomial degree satisfies \(k=k_g\ge 5\).
Then there exists \(h_0>0\) such that, for all \(0<h\le h_0\), the numerical scheme~\eqref{eq:num} admits a unique solution
\(\bigl(r_h(t),\partial_t r_h(t),\lambda_h(t)\bigr)\) with \(\lambda_h(t)\equiv 0\), which exists uniquely on \([0,T]\) and satisfies
\[
r_h \in C^1([0,T];V_h^3).
\]
Moreover, the following error estimate holds:
\[
\normt{r_h^{(\ell)}(t)-r(t)}_{\Man}
=\Big(\|r_h^{(\ell)}(t)-r(t)\|_{L^2(\Man)}^2
+\|{\rm D}_r\bigl(r_h^{(\ell)}(t)-r(t)\bigr)\|_{L^2(\Man)}^2\Big)^{1/2}
\;\lesssim\; h^{k},
\]
for all \(t\in[0,T]\).
\end{theorem}

\begin{remark}[Isometric embedding of the Ricci flow with error estimates]
    \upshape
Our results can be naturally applied to the finite element approximation of isometric embeddings of the Ricci flow.
Specifically, if the target metric \(g(t)\) on the right-hand side of \eqref{eq:isoemb} in Problem~\ref{Pro:dynamic} is taken to be a smooth solution of the two-dimensional Ricci flow, then, as discussed in Remark~\ref{rmk:posi_curv}, \(g(t)\) satisfies the assumptions of Theorem~\ref{thm:err}.

In the numerical scheme, we take \(g_h(t)\) on the right-hand side of \eqref{eq:num-a} to be the finite element approximation constructed by the method in \cite[Eq.~(3.3)]{gao2025ricci} (see also \cite[Eqs.~(9)--(11)]{gawlik2019finite}). 
In particular, one may employ a degree-\(k\) Regge finite element space to discretize the metric and a degree-\(k\) Lagrange finite element space to discretize the curvature. Then the error analysis in Theorem~3.1 of \cite{gao2025ricci} guarantees that the approximation property \eqref{eq:gh_app} holds. 

Consequently, the numerical solution \(r_h(t)\) produced by scheme \eqref{eq:num} yields a finite element approximation of the isometric embedding of the Ricci flow, with a rigorous error estimate provided by Theorem~\ref{thm:err}.
\end{remark}

\section{Discrete Korn inequality}\label{sec:korn}
In this section, we establish a discrete Korn inequality for finite element functions, which characterizes infinitesimal rigidity at the discrete level.
It is used to prove the existence and uniqueness of the numerical solution and is also a key ingredient in the convergence and error analysis.
 
\begin{lemma}[Discrete Korn inequality]\label{lm:dis_Korn}
Let \(r:\Man \rightarrow \mathbb{R}^3\) be a smooth embedding of $g$ with positive Gaussian curvature. Assume that the induced metric of \(\Man\) is approximated by Regge finite element space of degree \(k_g\ge2\). Then there exists \(h_0>0\) such that, for any \(0<h<h_0\) and any finite element function \(\tilde r_h\in V_h^3\) satisfying
\begin{equation}\label{eq:dis_Korn_cond} 
\| \tilde r_h^{(\ell)}-r\|_{W^{1,\infty}(\Man)}\;\lesssim\; h^{1+\varepsilon},\qquad \varepsilon\in(0,1),
\end{equation} 
it holds that
\begin{equation}\label{eq:dis_Korn} 
 \|v_h\|_{L^2(\Man_h)}
+ \|\mathrm{d}\tilde r_h \odot \mathrm{d}v_h\|_{L^2(\Man_h)}
\lesssim  \|\mathrm{d}\tilde r_h \odot \mathrm{d}v_h\|_{L^2(\Man_h)} 
, \quad \forall\, v_h \in \big(\RM[\tilde r_h]\big)^{\perp},
\end{equation} 
where \(\RM[\tilde r_h]=\{\alpha\times \tilde r_h+\beta:\alpha,\beta\in\mathbb{R}^3\}\), and \(\big(\RM[\tilde r_h]\big)^{\perp}\) denotes the \(L^2(\Man_h)\)-orthogonal complement with respect to the inner product \((\cdot,\cdot)_{\Man_h}\) defined in \eqref{eq:inner_Mh}.
\end{lemma}
\begin{proof}
To begin, we split $\| \mathrm{d} \tilde r_h \odot \mathrm{d} v_h \|_{L^2(\Man_h)} $
into three terms $I_1, I_2,$ and $I_3$; namely,  
\[ \begin{aligned}
    \| \mathrm{d} \tilde r_h \odot \mathrm{d} v_h \|_{L^2(\Man_h)}^2 
    & = \| \mathrm{d} r \odot \mathrm{d} v_h^{(\ell)} \|_{L^2(\Man)}^2 
    + \big(\| \mathrm{d} r^{(-\ell)} \odot \mathrm{d} v_h \|_{L^2(\Man_h)}^2  - \| \mathrm{d} r \odot \mathrm{d} v_h^{(\ell)} \|_{L^2(\Man)}^2 \big) \\ 
    &\qquad + \big(  \| \mathrm{d} \tilde r_h \odot \mathrm{d} v_h \|_{L^2(\Man_h)}^2  - \| \mathrm{d} r^{(-\ell)} \odot \mathrm{d} v_h \|_{L^2(\Man_h)}^2 \big) \\
    & =: I_1 + I_2 + I_3. 
\end{aligned} \]
    We further decompose \(v_h^{(\ell)}\) orthogonally as \(v_h^{(\ell)} = (v_h^{(\ell)})_{\perp} + (v_h^{(\ell)})_{\mathrm{RM}}\), where \((v_h^{(\ell)})_{\perp} \in (\mathrm{RM}[r])^{\perp}\) and \((v_h^{(\ell)})_{\mathrm{RM}} \in \mathrm{RM}[r]\).
    Then, it follows from the Korn inequality \eqref{eq:korn} and \eqref{eq:rigid} that  
    \[ \begin{aligned}
        I_1 & = \| {\rm D}_r  v_h^{(\ell)} \|_{L^2(\Man)}^2  
        = \| {\rm D}_r (v_h^{(\ell)})_{\perp} \|_{L^2(\Man)}^2 
        \\
        & \gtrsim \| (v_h^{(\ell)})_{\perp}\|_{L^2(\Man)}^2 = 
        \| v_h^{(\ell)}\|_{L^2(\Man)}^2 - \| (v_h^{(\ell)})_{\RM}\|_{L^2(\Man)}^2. 
    \end{aligned}\]
    Combining \eqref{eq:geo_err_tensor} and the inverse inequality, it holds that
    \[ \begin{aligned}
        |I_2| & = | \| \mathrm{d} r^{(-\ell)} \odot \mathrm{d} v_h \|_{L^2(\Man_h)}^2  - \| \mathrm{d} r \odot \mathrm{d} v_h^{(\ell)} \|_{L^2(\Man)}^2 |  \\ 
        & \lesssim h^{k_g+1} \| \mathrm{d} r^{(-\ell)} \odot  \mathrm{d} v_h\|_{L^2(\Man_h)}^2 \lesssim h^{k_g+1} \|\mathrm{d} v_h\|_{L^2(\Man_h)}^2 \lesssim h^{k_g - 1} \| v_h\|_{L^2(\Man_h)}^2. 
    \end{aligned} \]
    By the condition \eqref{eq:dis_Korn_cond}, we have 
    \[ |I_3| 
    \lesssim \| \mathrm{d} (\tilde r_h - r^{(-\ell)}) \odot \mathrm{d} v_h \|_{L^2(\Man_h)}^2
    \lesssim h^{2\varepsilon + 2} \|\mathrm{d} v_h\|_{L^2(\Man_h)}^2 \lesssim 
    h^{2\varepsilon} \| v_h\|_{L^2(\Man_h)}^2 .
    \]
    Summing the above estimates for $I_1, I_2, I_3$, we deduce 
    \[ \begin{aligned}
        \| \mathrm{d} \tilde r_h \odot \mathrm{d} v_h \|_{L^2(\Man_h)}^2 &\gtrsim
        (1 - h^{k_g-1} - h^{2\varepsilon})\|v_h\|_{L^2(\Man_h)}^2 -  \| (v_h^{(\ell)})_{\RM}\|_{L^2(\Man)}^2\\ 
        &\gtrsim \|v_h\|_{L^2(\Man_h)}^2 -  \| (v_h^{(\ell)})_{\RM}\|_{L^2(\Man)}^2,
    \end{aligned} \]
    where we have used the norm equivalence \eqref{eq:norm-equi} and the condition $k_g \geq 2$. 
    Now, it remains to estimate $\| (v_h^{(\ell)})_{\RM}\|_{L^2(\Man)}^2$.
    Exploiting $v_h \in (\RM[\tilde r_h])^\perp$ 
    and the norm equivalence in finite-dimensional spaces $\RM[r]$, we deduce 
    \[ \begin{aligned}
       \| (v_h^{(\ell)})_{\RM}\|_{L^2(\Man)} & \eqsim 
       \sup_{\alpha, \beta \in \mathbb{R}^3} 
       \frac{((v_h^{(\ell)})_{\RM}, 
       \alpha \times r + \beta)_{\Man}}{|\alpha| + |\beta|} 
       = \sup_{\alpha, \beta \in \mathbb{R}^3} 
       \frac{(v_h^{(\ell)}, 
       \alpha \times r + \beta)_{\Man} }{|\alpha| + |\beta|}  \\ 
       & = \sup_{\alpha, \beta \in \mathbb{R}^3} 
       \frac{(v_h^{(\ell)}, 
       \alpha \times r + \beta)_{\Man} 
       - (v_h, \alpha \times \tilde r_h + \beta)_{\Man_h} 
       }{|\alpha| + |\beta|} \lesssim h^{1+\varepsilon} \|v_h\|_{L^2(\Man_h)},
    \end{aligned} \] 
where in the last step, we have used \eqref{eq:geo_err_vector} and the condition \eqref{eq:dis_Korn_cond}.    
By substituting this estimate, we conclude the proof.
\end{proof}
\begin{remark}[Discrete infinitesimal rigidity]\label{rmk:rigid}
\upshape
As a corollary of Lemma~\ref{lm:dis_Korn}, we have
\begin{equation}\label{eq:dis-rigidity}
\RM[\tilde r_h] \;=\; \{\, v_h \in V_h^3 \mid \mathrm{d}\tilde r_h \odot \mathrm{d}v_h = 0 \,\},
\end{equation}
for any $\tilde r_h$ satisfying \eqref{eq:dis_Korn_cond}. Indeed, if $v_h=\alpha\times \tilde r_h+\beta\in \RM[\tilde r_h]$ with $\alpha,\beta\in\mathbb{R}^3$, then by direct computation, we have
\[
\mathrm{d}\tilde r_h \odot \mathrm{d}v_h
= \tfrac12(\partial_i\tilde r_h\cdot(\alpha\times \partial_j\tilde r_h)
+ (\alpha\times \partial_i\tilde r_h)\cdot \partial_j\tilde r_h\big)\,\mathrm{d}x^i\otimes \mathrm{d}x^j
=0.
\]
Conversely, if $v_h\in V_h^3$ satisfies $\mathrm{d}\tilde r_h\odot \mathrm{d}v_h=0$, decompose it $L^2(\Man_h)$–orthogonally (with respect to $(\cdot,\cdot)_{\Man_h}$) as $v_h=v_{h,\mathrm{RM}}+v_{h,\perp}$, where $v_{h,\mathrm{RM}}\in \RM[\tilde r_h]$ and $v_{h,\perp}\in (\RM[\tilde r_h])^\perp$. Lemma~\ref{lm:dis_Korn} then yields
\[
0=\|\mathrm{d}\tilde r_h\odot \mathrm{d}v_h\|_{L^2(\Man_h)}
=\|\mathrm{d}\tilde r_h\odot \mathrm{d}v_{h,\perp}\|_{L^2(\Man_h)} \gtrsim \|v_{h,\perp}\|_{L^2(\Man_h)},
\]
so $v_{h,\perp}=0$ and thus $v_h=v_{h,\mathrm{RM}}\in \RM[\tilde r_h]$.
\end{remark}
In the following context, \( \tilde{r}_h \) in Lemma \ref{lm:dis_Korn} can be selected as the numerical solution \( r_h \) to demonstrate the well-posedness of the numerical scheme in Section \ref{subsec:wellpose_num}, or it may alternatively be selected as the interpolant of the smooth solution, as described in Section \ref{subsec:dis_graph_norm}, for the purpose of estimating the error.

\section{Well-posedness and error estimates of numerical scheme}\label{sec:pf_main}
This section focuses on the proof of our main result, Theorem~\ref{thm:err}, which concerns the well-posedness of the numerical scheme \eqref{eq:num} as well as its error estimates.
Let \(I_h\) be the Lagrange interpolation operator onto \(V_h\), and define the interpolant
\[r_h^*(t)=I_h r^{(-\ell)}(t)\in V_h^3,\] 
which, by standard finite element approximation theory,
provides a quasi-optimal approximation of the smooth solution \(r(t)\).
We choose the initial data in \eqref{eq:num} as
\begin{equation}\label{eq:ini_data}
r_h(0)=r_h^*(0)=I_h r^{(-\ell)}(0).
\end{equation}
Throughout, we fix the polynomial degree $k=k_g\ge 5$.
The choice of a high polynomial degree is essential for controlling the \(W^{1,\infty}\)- and \(H^1\)-norms
of the errors arising from the nonlinear terms in the error equation.
With this, inverse inequalities can be applied to reduce these contributions to \(L^2\)-norms, ensuring that they remain sufficiently small.

\subsection{Well-posedness of the numerical scheme}\label{subsec:wellpose_num}

In this subsection, we recast the numerical scheme \eqref{eq:num} as an equivalent ODE system and prove its
local existence, uniqueness, and continuous dependence on \(t\).
We begin by considering a fixed \(r_h \in V_h^3\). Then \eqref{eq:num-a}--\eqref{eq:num-b} can be viewed as a linear system for
\((v_h,\lambda_h)\in V_h^3\times \RM[r_h]\).
We define
\begin{align*}
a_h(u_h, q_h) &= 2(\mathrm{d}r_h \odot \mathrm{d}u_h, \mathrm{d}r_h \odot \mathrm{d}q_h)_{\Man_h},
\quad && \forall u_h, q_h \in V_h^3,
\\
b_h(q_h, \lambda_h) &= (q_h,\lambda_h)_{\Man_h}, \quad && \forall q_h \in V_h^3
\text{ and } \lambda_h \in \RM[r_h].
\end{align*}
Now suppose that \(r_h\in V_h^3\) satisfies \eqref{eq:dis_Korn_cond}. Then Lemma~\ref{lm:dis_Korn} implies that
\[
a_h(u_h,u_h)=\|\mathrm{d} r_h \odot \mathrm{d} u_h\|_{L^2(\Man_h)}^{2}\;\gtrsim\;\|u_h\|_{L^2(\Man_h)}^{2},
\qquad \forall\,u_h\in(\RM[r_h])^{\perp}.
\]
Furthermore, since \(\RM[r_h] \subset V_h^3 \), we have
\[
    \inf_{\lambda_h \in \RM[r_h]}
    \sup_{q_h \in V_h^3}
    \frac{
        b_h(q_h, \lambda_h)
    }{
        \|\lambda_h\|_{L^2(\Man_h)}
        \|q_h\|_{L^2(\Man_h)}
    } \geq
    \inf_{\lambda_h \in \RM[r_h]}
    \frac{
        b_h(\lambda_h, \lambda_h)
    }{
        \|\lambda_h\|_{L^2(\Man_h)}
        \|\lambda_h\|_{L^2(\Man_h)}}
    = 1.
\]
By the abstract  theory in \cite{boffi2013mixed}, the two estimates above (coercivity on the kernel and an inf--sup condition)
imply that \eqref{eq:num-a}--\eqref{eq:num-b} admits a unique solution \((v_h,\lambda_h)\in V_h^3\times \RM[r_h]\),
provided that \(r_h\) satisfies \eqref{eq:dis_Korn_cond}. 
In fact, we have \(\lambda_h=0\), since we can test \eqref{eq:num-a} with \(q_h=\lambda_h\in \RM[r_h]\subset V_h^3\) and invoke the discrete infinitesimal rigidity \eqref{eq:dis-rigidity}. 
Moreover, the abstract theory in \cite{boffi2013mixed} also yields the stability estimate
\[
\| v_h \|_{L^2(\Man_h)} + \|\lambda_h\|_{L^2(\Man_h)} \;\le\; C_{\mathbf r}\,\|\partial_t g_h\|_{L^2(\Man_h)},
\]
where the constant \(C_{\mathbf r}\) depends Lipschitz continuously on the norm of \(r_h\in V_h^3\).

We collect the nodal values of \(r_h\in V_h^3\) into a vector \(\mathbf r \in \mathbb{R}^{\dim V_h^3}\),
and denote by \(\mathbf f(t,\mathbf r)\in \mathbb{R}^{\dim V_h^3}\) the vector of nodal values of the velocity \(v_h\) obtained as the solution of the linear system \eqref{eq:num-a}--\eqref{eq:num-b}.
The discussion above indeed shows that, for any \(\mathbf r \in \mathbb{R}^{\dim V_h^3}\) whose associated finite element function
\(r_h\in V_h^3\) satisfies \eqref{eq:dis_Korn_cond}, the mapping \(\mathbf f(t,\mathbf r)\in \mathbb{R}^{\dim V_h^3}\)
is well defined and depends Lipschitz continuously on \(\mathbf r\).
The numerical scheme \eqref{eq:num} is equivalent to the following ODE system:
\begin{equation}\label{eq:ode_num}
\frac{\mathrm d \mathbf r}{\mathrm d t}(t)=\mathbf f\bigl(t,\mathbf r(t)\bigr).
\end{equation}
This ODE is equipped with the initial value \(\mathbf r(0)\) chosen as in \eqref{eq:ini_data}.
By the interpolation error estimates and the assumption \(k\ge 5\), the initial value \(r_h(0)\) satisfies \eqref{eq:dis_Korn_cond}.
Therefore, the ODE \eqref{eq:ode_num} (equivalently, the scheme \eqref{eq:num}) admits a unique local solution \(r_h(t) \in V_h^3\)
that depends continuously on \(t\).

\subsection{Discrete graph norm}\label{subsec:dis_graph_norm}
We define the discrete graph norm associated with \( r_h^* \) as
\begin{equation}\label{eq:dis_graph_norm}
\normt{v_h}_h^2 = \|v_h\|_{L^2(\Man_h)}^2 + \| \mathrm{d} r_h^* \odot \mathrm{d} v_h \|_{L^2(\Man_h)}^2,
\qquad \forall\, v_h \in V_h^3.
\end{equation}
The dual norm is defined by
\[
\normt{w_h}_{h,*} = \sup_{v_h \in V_h^3 \setminus \{0\}} 
\frac{  (w_h, v_h)_{\Man_h} }{ \normt{v_h}_h },
\qquad \forall\, w_h \in V_h^3.
\]
These norms will be employed in the following stability and error estimates.
By the approximation property of \( r_h^* \) (and for sufficiently small \(h\)), we have that \(r_h^*\) satisfies \eqref{eq:dis_Korn_cond},
and hence Lemma~\ref{lm:dis_Korn} implies the discrete Korn inequality
\begin{equation}\label{eq:dis_korn_rh*}
\| \mathrm{d} r_h^* \odot \mathrm{d} v_h \|_{L^2(\Man_h)} \gtrsim \| v_h \|_{L^2(\Man_h)},
\quad \forall v_h \in \big(\RM[r_h^*]\big)^{\perp},
\end{equation}
where \(\RM[r_h^*]=\{\alpha\times r_h^*+\beta:\alpha,\beta\in\mathbb{R}^3\}\), and \(\big(\RM[r_h^*]\big)^{\perp}\) denotes the \(L^2(\Man_h)\)-orthogonal complement with respect to the inner product \((\cdot,\cdot)_{\Man_h}\) defined in \eqref{eq:inner_Mh}. 
We use \( \Pi_{\RM}: V_h^3 \to \RM[r_h^*] \) to denote the \( L^2(\Man_h) \)-orthogonal projection, given by
\begin{equation}\label{eq:Pi_RM}
( \Pi_{\RM} v_h, \alpha \times r_h^* + \beta )_{\Man_h}
= ( v_h, \alpha \times r_h^* + \beta )_{\Man_h},
\quad \forall\, \alpha, \beta \in \mathbb{R}^3.
\end{equation}
We emphasize that the orthogonal projection $\Pi_{\RM}$ implicitly depends on the discrete map 
$r_h^* \in V_h^3$, although this dependence is not explicitly shown in the notation. 
As a direct consequence of the discrete Korn inequality \eqref{eq:dis_korn_rh*} and the rigidity property \eqref{eq:dis-rigidity}, we obtain the following estimate:
\begin{equation}\label{eq:graphnorm_esti}
\normt{v_h}_h \leq \normt{\Pi_{\RM} v_h}_h + \normt{ v_h - \Pi_{\RM} v_h}_h
\lesssim \|\Pi_{\RM} v_h \|_{L^2(\Man_h)} +  
\|\mathrm{d} r_h^* \odot \mathrm{d} v_h\|_{L^2(\Man_h)},
\end{equation}
for all \( v_h \in V_h^3 \).
In the following stability estimate, the $W^{1,\infty}(\Man_h)$ or $H^1(\Man_h)$ norms of the error may appear in the nonlinear terms. To handle this, we apply the following inverse inequality to convert these terms into $L^2(\Man_h)$ norms, which are then controlled by the discrete graph norm:
\begin{equation}\label{eq:inv}
\|v_h\|_{W^{1,\infty}(\Man_h)} \lesssim h^{-1}\,\|v_h\|_{H^1(\Man_h)} \lesssim h^{-2}\,\|v_h\|_{L^2(\Man_h)} \leq h^{-2}\,\normt{v_h}_h, \quad \forall v_h \in V_h^3.
\end{equation}


\subsection{Defect estimates}
We define the defect functions \(d_v(t) \in V_h^3\) and \(d_{\lambda}(t) \in \RM[r_h^*]\) as follows:
\begin{subequations}\label{eq:def}  
\begin{align}
\label{eq:def-a} 
(d_v, q_h)_{\Man_h}
&= 
(\,\partial_t g_h - 2\,\mathrm{d} r_h^* \odot \mathrm{d}(\partial_t r_h^*)\,,\, \mathrm{d} r_h^* \odot \mathrm{d} q_h\,)_{\Man_h}, && \forall  q_h \in V_h^3 \\ 
\label{eq:def-b} 
(d_{\lambda}, \mu_h^*)_{\Man_h} &= 
(\partial_t r_h^*, \mu_h^*)_{\Man_h}, && \forall \mu_h^* \in \RM[r_h^*].
\end{align}
\end{subequations}
The approximation property of $r_h^*(t)$ and \eqref{eq:gh_app} guarantee the following defect estimates: 
\begin{lemma}[Defect estimates]\label{lm:def-esti} 
Under the conditions of Theorem \ref{thm:err}, for the defect functions \(d_v(t) \in V_h^3\) and \(d_{\lambda}(t) \in \RM[r_h^*]\) defined in \eqref{eq:def}, it holds that 
\[ \normt{d_v(t)}_{h,*} \lesssim h^k, \quad \text{and} \quad \|d_\lambda(t)\|_{L^2(\Man_h)} \lesssim h^{k+1}, \qquad \forall t \in [0,T]. \]
\end{lemma} 
\begin{proof}
Substituting the second identity in \eqref{eq:g_ell} into \eqref{eq:def-a}, we have
\[ 
    (d_v, q_h)_{\Man_h} = 
     (\partial_t g_h - \partial_t g^{(-\ell)} , \mathrm{d} r_h^* \odot \mathrm{d} q_h)_{\Man_h} -
    2(\mathrm{d} r_h^* \odot \mathrm{d} \partial_t r_h^* - \mathrm{d} r^{(-\ell)} \odot \mathrm{d} \partial_t r^{(-\ell)}, \mathrm{d} r_h^* \odot \mathrm{d} q_h)_{\Man_h}.  
\]
By using the approximation property of $r_h^*$ and the estimate \eqref{eq:gh_app}, 
it holds that
\[ (d_v, q_h)_{\Man_h} \lesssim h^k \|\mathrm{d} r_h^* \odot \mathrm{d} q_h \|_{L^2(\Man_h)}
\leq h^k \normt{q_h}_h
. \]
Writing $ \mu_h^* = \mu_h^{(1)} \times r_h^* + \mu_h^{(2)} $ with $\mu_h^{(1)}, \mu_h^{(2)} \in \mathbb{R}^3$ and  substituting \eqref{eq:ptr_orth_RM} into \eqref{eq:def-b}, we obtain
\[ \begin{aligned}
    (d_{\lambda}, \mu_h^*)_{\Man_h} &= 
    (\partial_t r_h^*, \mu_h^{(1)} \times r_h^* + \mu_h^{(2)})_{\Man_h} - 
    (\partial_t r, \mu_h^{(1)} \times r + \mu_h^{(2)})_{\Man} \\ 
     &\lesssim   (\partial_t r_h^*, \mu_h^{(1)} \times r_h^* + \mu_h^{(2)})_{\Man_h} - 
    (\partial_t r^{(-\ell)}, \mu_h^{(1)} \times r^{(-\ell)} + \mu_h^{(2)})_{\Man_h} + h^{k_g+1}\|\mu_h^*\|_{L^2(\Man_h)} \\
    &\lesssim h^{k+1}\|\mu_h^*\|_{L^2(\Man_h)}, 
\end{aligned} \]
where we have used \eqref{eq:geo_err_vector} together with the approximation property of $r_h^*$. 
\end{proof}

\subsection{Error equation and continuity method}
We define the finite element error functions as 
\begin{align*}
    \hspace{-3mm}
{e_r}(t)  & = r_h(t) - r_h^*(t) \in V_h^3,  \quad  \text{and} \quad
e_v(t)  = \partial_t e_r(t)  = \partial_t r_h(t) - \partial_t r_h^*(t) \in V_h^3. 
\end{align*}
From the initial data \eqref{eq:ini_data}, one immediately obtains $e_r(0)=0$. 
To derive the error equation, we compare the numerical scheme \eqref{eq:num} with the defect equation \eqref{eq:def}, and use the following identities, 
\begin{align}
2\, \mathrm{d}r_h \odot  \mathrm{d} {\partial_t} r_h - 2\, \mathrm{d} r_h^* \odot \mathrm{d} {\partial_t}r_h^* 
&= \partial_t \big( \mathrm{d}r_h \odot \mathrm{d}r_h - \mathrm{d}r_h^* \odot \mathrm{d}r_h^* \big) \notag \\ 
&= \partial_t \big( \mathrm{d}r_h \odot \mathrm{d} e_r + \mathrm{d}e_r \odot \mathrm{d}r_h^* \big) 
= \partial_t \big( 2\, \mathrm{d}r_h^* \odot \mathrm{d}e_r + \mathrm{d}e_r \odot \mathrm{d}e_r \big), \notag \\
\big( \partial_t r_h, \mu_h )_{\Man_h} - \big( \partial_t r_h^*, \mu_h^* )_{\Man_h}
&= \big( \partial_t e_r, \mu_h^* )_{\Man_h} + \big( \partial_t r_h, \mu_h^{(1)} \times e_r )_{\Man_h},\notag
\end{align}
where  $\mu_h = \mu_h^{(1)} \times r_h + \mu_h^{(2)} \in \RM[r_h]$ and $\mu_h^* = \mu_h^{(1)} \times r_h^* + \mu_h^{(2)} \in \RM[r_h^*]$. 
Hence, the error equation is obtained by subtracting \eqref{eq:def} from \eqref{eq:num}, applying the identities above, and noting that \(\lambda_h=0\):
\begin{subequations}\label{eq:err}  
\begin{align}
2(\partial_t(   \mathrm{d} r_h^* \odot\mathrm{d} e_r), \mathrm{d} r_h^* \odot \mathrm{d} q_h)_{\Man_h} 
 &= ( \partial_t g_h - 2\, \mathrm{d}r_h \odot \mathrm{d} \partial_t r_h,\ \mathrm{d}e_r \odot \mathrm{d}q_h )_{\Man_h}  \notag \\
\label{eq:err-a} 
& \quad
- ( \partial_t (\mathrm{d}e_r \odot \mathrm{d}e_r),\ \mathrm{d}r_h^* \odot \mathrm{d}q_h )_{\Man_h} - (d_v, q_h)_{\Man_h}, \\
\label{eq:err-b} 
( \partial_t e_r, \mu_h^* )_{\Man_h} &= - (\partial_t r_h, \mu_h^{(1)} \times e_r)_{\Man_h} - (d_{\lambda}, \mu_h^*)_{\Man_h},
\end{align}
\end{subequations}
for all test functions  \( q_h \in V_h^3 \) and \( \mu_h^* = \mu^{(1)}_h \times r_h^* + \mu^{(2)}_h \in \RM[r_h^*] \) with $\mu_h^{(1)}, \mu_h^{(2)} \in \mathbb{R}^3$.

\smallskip

We focus on the estimate of the norm
\[
\normt{e_r(t)}_h  = (\|e_r( t)\|_{L^2(\Man_h)}^2 + \|\mathrm{d} r_h^*(t) \odot \mathrm{d} e_r(t)\|_{L^2(\Man_h)}^2)^{1/2}. 
\] 
The continuity method for semi-discretization scheme \eqref{eq:num} is to consider 
\begin{equation}\label{eq:t*} 
\begin{split}
    t^*=&\sup\left\{ t \in [0, T]: \eqref{eq:num} \text{ has a solution on } [0,t],  \right. \\
     &\quad\quad \text{and } \normt{e_r(s)}_h \leq h^{4 + \varepsilon}, \text{ for } 0 \leq s \leq t\ \},
\end{split} 
\end{equation}
where $0<\varepsilon<1$ is a small positive number. 
Since \( e_r(0) = 0 \) and the semi-discrete finite element solution of \eqref{eq:num} locally exists  and depends continuously on $t$, it follows that \( t^* > 0 \). Assume \( t^* < T\), in subsequent subsections we will show the solution $r_h$ of \eqref{eq:num} can be extended in a neighborhood of $t^*$ with the same estimate, then by contradiction we get \( t^* = T\) and complete the proof of Theorem \ref{thm:err}. The following estimate can be obtained from \eqref{eq:t*} and the inverse inequality \eqref{eq:inv}
\begin{equation}\label{eq:t*:cor}
    \|  e_{r}(s) \|_{W^{1,\infty}(\Man_h)}\lesssim h^{2+\varepsilon}, \quad s \in [0,t^*].
\end{equation}

\begin{remark}
\upshape
The powers of \(h\) in \eqref{eq:t*} and \eqref{eq:t*:cor} are needed to control the nonlinear terms arising in the error analysis.
For instance, \eqref{eq:t*:cor} is used in Lemma~\ref{lm:esti_ev} to estimate \(J_1\) and \(J_2\), and to control
\(\|\Pi_{\RM} e_v\|_{L^2(\Man_h)}\).
Likewise, \eqref{eq:t*} and \eqref{eq:t*:cor} are used in \eqref{eq:boundedness} to prove a uniform \(W^{1,\infty}\)-bound for the numerical solution.

Moreover, \eqref{eq:t*} is used in Proposition~\ref{prop:esti_drder} to estimate \(K_1\) and \(K_2\), and here the exponent of \(h\) in \eqref{eq:t*} must be strictly larger than \(4\) to absorb the constants appearing in the estimates of \(K_1\) and \(K_2\).
This is precisely where high-order finite elements of degree \(k\ge 5\) are required, in view of the error estimate in Theorem~\ref{thm:err}.
\end{remark}

\subsection{Time derivative of the error}

We are now ready to estimate the time derivative of the error, i.e., $e_{v} = \partial_t e_r$
by choosing the test functions \( q_h = e_{v} \) in \eqref{eq:err-a} and \( \mu_h^* = \Pi_{\RM} e_{v} \) in \eqref{eq:err-b}. 
\begin{lemma}[Estimate of \( e_v \)]\label{lm:esti_ev}
    Under the assumptions of Theorem \ref{thm:err}, there exists a constant $h_0 > 0$ such that the following estimate holds uniformly for $0<h\leq h_0$ and $t\in[0,t^{*}]${\rm:}
    \[
    \normt{e_v(t) }_h\lesssim h^{-1} \normt{e_r(t)}_h + \normt{d_v(t)}_{h,*} + \|d_\lambda(t)\|_{L^2(\Man_h)} + h^{k+1}.
    \]
\end{lemma}
\begin{proof}
Decomposing 
$ \partial_t( \mathrm{d} r_h^* \odot\mathrm{d} e_r) =  \mathrm{d} r_h^* \odot\mathrm{d} e_v 
+  \mathrm{d} \partial_t r_h^* \odot\mathrm{d} e_r $ and choosing $q_h = e_{v}$ in \eqref{eq:err-a}, 
it holds that
\[ 
\begin{aligned}
    2\|\mathrm{d} e_{v} \odot \mathrm{d} r_h^{*}\|_{L^2(\Man_h)}^2 
&=(\partial_t g_h - 2 \mathrm{d} r_h \odot \mathrm{d} \partial_t r_h, \mathrm{d} e_{r} \odot \mathrm{d} e_{v} )_{\Man_h}
- 2 ( \mathrm{d} e_{r} \odot \mathrm{d} e_{v}, \mathrm{d} r_h^{*} \odot \mathrm{d} e_{v} )_{\Man_h} \\
&\quad  - ( d_v, e_{v} )_{\Man_h}
- 2 ( \mathrm{d} e_{r} \odot \mathrm{d} \partial_t r_h^{*}, \mathrm{d} r_h^{*} \odot \mathrm{d} e_{v})_{\Man_h}\\ 
& =: J_1 + J_2 + J_3 + J_4.
\end{aligned}
\]
For the term $J_1$, we first use \eqref{eq:g_ell} and \eqref{eq:gh_app}, together with the approximation property of $r_h^*$ to deduce that 
\begin{equation}\label{eq:pt_gh-drhdrh} 
\begin{aligned}
   & \| \partial_t g_h - 2 \mathrm{d} r_h \odot \mathrm{d} \partial_t r_h \|_{L^2(\Man_h)} \\ 
 \lesssim &\, \|\mathrm{d}r_h^*\odot\mathrm{d}\partial_t r_h^* - \mathrm{d}r_h\odot\mathrm{d}\partial_t r_h\|_{L^2(\Man_h)} + \| \partial_t g_h -  \partial_t g^{(-\ell)}  \|_{L^2(\Man_h)} \\ 
 & + \| \mathrm{d}r^{(-\ell)}\odot\mathrm{d}\partial_t r^{(-\ell)}
 - \mathrm{d} r_h^* \odot \mathrm{d} \partial_t r_h^*\|_{L^2(\Man_h)} \\ 
\lesssim &\,  \| \mathrm{d} r_h^{*} \odot \mathrm{d} e_{v}  \|_{L^2(\Man_h)}
+ \|\mathrm{d} e_{r} \odot \mathrm{d} e_{v} \|_{L^2(\Man_h)} 
+ \|\mathrm{d} e_{r} \odot \mathrm{d} \partial_t r_h^{*}\|_{L^2(\Man_h)} + h^k \\ 
\lesssim & 
\, \normt{ e_v }_{h} + \| e_r \|_{W^{1,\infty}(\Man_h)} \| e_v \|_{H^1(\Man_h)} 
+ \| e_r \|_{H^1(\Man_h)} + h^k\\
\lesssim & \, 
(1 + h^{-1}\|e_r\|_{W^{1, \infty}(\Man_h)})\normt{ e_v }_{h} + h^{-1} \normt{e_r}_h + h^k 
\quad\quad\quad\quad (\text{by inverse estimate \eqref{eq:inv}})
\\ 
\lesssim &
\, (1 + h^{1+\varepsilon})\normt{ e_v }_{h} + h^{-1} \normt{e_r}_h + h^k,
\qquad\qquad\qquad\qquad\qquad
 (\text{by estimate \eqref{eq:t*:cor}}). 
\end{aligned}
\end{equation} 
Substituting the estimate \eqref{eq:pt_gh-drhdrh}, we obtain the estimate for $J_1$: 
\[ \begin{aligned}
    J_1 &\le  
\|  \mathrm{d} e_{r} \odot \mathrm{d} e_{v} \|_{L^2(\Man_h)}\| \partial_t g_h - 2 \mathrm{d} r_h \odot \mathrm{d} \partial_t r_h \|_{L^2(\Man_h)}\\
&\lesssim \|  e_{r} \|_{W^{1,\infty}(\Man_h)} \| e_v \|_{H^1(\Man_h)}
(\normt{ e_v }_{h} + h^{-1} \normt{e_r}_h + h^k ) \\ 
&\lesssim h^{-1}\|  e_{r} \|_{W^{1,\infty}(\Man_h)} \normt{e_v}_h 
(\normt{ e_v }_{h} + h^{-1} \normt{e_r}_h + h^k)
\qquad \text{(by inverse estimate \eqref{eq:inv})}
\\ 
&\lesssim h^{1+\varepsilon} \normt{e_v}_h
(\normt{ e_v }_{h} + h^{-1} \normt{e_r}_h + h^k )
\qquad\qquad\qquad\qquad \text{(by estimate \eqref{eq:t*:cor})}\\
& \leq (Ch^{1+\varepsilon} + \delta ) \normt{e_v}_h^2 
+ C_{\delta}( h^{2\varepsilon} \normt{e_r}_h^2 + h^{2k+2+2\varepsilon} ),
\end{aligned}
\]
where $\delta>0$ can be chosen arbitrarily small, and $C_{\delta}>0$ is a constant (independent of $h$) depending only on $\delta$. 
Similarly, by the inverse estimate \eqref{eq:inv} and \eqref{eq:t*:cor}, we obtain
\[
J_2 \lesssim
\| e_r \|_{W^{1,\infty}(\Man_h)} \| e_v \|_{H^1(\Man_h)} \normt{e_v}_h
\lesssim h^{1+\varepsilon} \normt{e_v}_h^2.
\]
For the estimates of \(J_3\) and \(J_4\), we have
\begin{subequations}
\begin{align}
J_3 & \leq \delta \normt{e_v}_h^2 + C_{\delta} \normt{d_v}_{h,*}^2, \notag \\
J_4 &\lesssim \|e_r\|_{H^1(\Man_h)}\normt{e_v}_h
\leq \delta \normt{e_v}_h^2 +
C_{\delta}h^{-2}\normt{e_r}_h^2. \notag
\end{align}
\end{subequations}
Collecting the bounds for $J_1,\ldots,J_4$, we have
\begin{equation}\label{eq:devdr}
    \|\mathrm{d} e_{v} \odot \mathrm{d} r_h^{*}\|_{L^2(\Man_h)}^2
\leq 
C_{\delta}\big( 
h^{-2}\normt{e_r}_h^2 + \normt{ d_v }_{h,*}^2 + h^{2(k+1)}\big) 
+ C(\delta + h^{1+\varepsilon}) \normt{e_v}^2_h.
\end{equation}
Taking $\mu_h^* = \Pi_{\RM} e_v = e_{v, \RM}^{(1)} \times r_h^* +  e_{v, \RM}^{(2)}$ in \eqref{eq:err-b}, we obtain 
\[ 
    (e_v, \Pi_{\RM} e_v)_{\Man_h} 
     = - (\partial_t r_h^*,  e_{v, \RM}^{(1)} \times e_r)_{\Man_h} 
    - (e_v, e_{v, \RM}^{(1)} \times e_r) _{\Man_h}
    - (d_{\lambda}, \Pi_{\RM} e_v)_{\Man_h}. 
\]
Since $\Pi_{\RM}: V_h^3 \to \RM[r_h^*]$ is the $L^2$-orthogonal projection, we deduce that
\begin{equation}\label{eq:e_v,RM}
\begin{aligned}
        \|\Pi_{\RM} e_v \|_{L^2(\Man_h)}^2 & \lesssim 
        \| e_v \|_{L^2(\Man_h)} \| e_r \|_{L^2(\Man_h)} + 
        \| e_r \|_{L^\infty(\Man_h)} \| e_v \|_{L^2(\Man_h)}^2\\
        & \quad + \| e_v \|_{L^2(\Man_h)} \| d_{\lambda} \|_{L^2(\Man_h)} \\ 
        & \leq C_{\delta} (\normt{e_r}_h^2 + \| d_{\lambda} \|_{L^2(\Man_h)}^2) + C(\delta + h^{3+ \varepsilon}) \normt{e_v}_h^2,
\end{aligned}
\end{equation}
where we have used  \eqref{eq:t*} and  the inverse estimate  $ \| e_r \|_{L^\infty(\Man_h)} \lesssim h^{-1} \| e_r \|_{L^2(\Man_h)} \lesssim h^{3 + \varepsilon} $.
Combining \eqref{eq:graphnorm_esti},  \eqref{eq:devdr} and \eqref{eq:e_v,RM}, we have 
\[ \begin{aligned}
    \normt{e_v}^2_h &\lesssim 
    \|\Pi_{\RM} e_v\|_{L^2(\Man_h)}^2 +  \|\mathrm{d} e_{v} \odot \mathrm{d} r_h^{*}\|_{L^2(\Man_h)}^2 \\ 
    & \leq C_{\delta}\big( h^{-2}\normt{e_r}_h^2  + \| d_{\lambda} \|_{L^2(\Man_h)}^2 + \normt{d_v}_{h,*}^2 + h^{2(k+1)} \big)
 + C(\delta + h^{1+\varepsilon}) \normt{e_v}_h^2. 
\end{aligned} \]
Noting that, for sufficiently small $h$ and $\delta$, we can choose them so that
$C(\delta + h^{1+\varepsilon}) \leq \tfrac12$, the term
$C(\delta + h^{1+\varepsilon})\,\normt{e_v}_h^2$ can be absorbed into the left-hand side.
This completes the proof of Lemma \ref{lm:esti_ev}.
\end{proof}

With the decomposition \( r_h = e_r + r_h^* \) and \( \partial_t r_h = e_v + \partial_t r_h^* \), we combine Lemma~\ref{lm:esti_ev} with \eqref{eq:t*} and the inverse inequality \eqref{eq:inv} to conclude that the numerical solution is bounded for sufficiently small \(h\):
\begin{equation}\label{eq:boundedness}
\max\big\{ \|r_h(t)\|_{W^{1,\infty}(\Man_h)}, \|\partial_t r_h(t)\|_{W^{1,\infty}(\Man_h)}\big\} \leq C, 
\quad t \in [0, t^*].
\end{equation}

\subsection{Stability estimates}
This subsection establishes the stability estimates for the error equation \eqref{eq:err} by 
estimating $\|\mathrm{d} e_{r} \odot \mathrm{d} r_h^*\|_{L^2(\Man_h)}$ and  $\|\Pi_{\RM} e_r\|_{L^2(\Man_h)}$ individually.

\begin{proposition}\label{prop:esti_drder}
    Under the assumptions of Theorem \ref{thm:err}, there exists a constant $h_0 > 0$ such that the following estimate holds uniformly for $0<h\leq h_0$ and $t\in[0,t^{*}]${\rm:}
    \begin{equation}\label{eq:esti_drder} 
        \|(\mathrm{d} e_{r} \odot \mathrm{d} r_h^*)(t)\|^2_{L^2(\Man_h)} 
     \lesssim 
    \int_0^t \big(\normt{e_r(s)}^2_h + \| d_\lambda(s) \|_{L^2(\Man_h)}^2 + \normt{d_v(s)}_{h,*}^2 \big)\mathrm{d}s
    + h^{2k+2}.
    \end{equation} 

\end{proposition}
\begin{proof}

With \( q_h = e_{r} \) in \eqref{eq:err-a}, we have  
\[
\begin{aligned}
     \frac{\mathrm{d}}{\mathrm{d}t}\|\mathrm{d} e_{r} \odot \mathrm{d} r_h^*\|_{L^2(\Man_h)}^2 
& = ( \partial_t g_h - 2 \mathrm{d}r_h \odot \mathrm{d} \partial_t r_h, \mathrm{d}e_r \odot \mathrm{d}e_{r}   )_{\Man_h} 
\\
& \quad -  2 ( \mathrm{d}e_r \odot \mathrm{d}e_v,\ \mathrm{d}r_h^* \odot \mathrm{d}e_{r}  )_{\Man_h} 
- (d_v, e_{r})_{\Man_h}  
\\
&=: K_1 + K_2 + K_3. 
\end{aligned}\]
For the term \(K_1\), using \eqref{eq:pt_gh-drhdrh}, we obtain
\[
\begin{aligned}
K_1 &
\leq \|\mathrm{d}e_r \odot \mathrm{d}e_{r} \|_{L^2(\Man_h)}
\|  \partial_t g_h - 2 \mathrm{d}r_h \odot \mathrm{d} \partial_t r_h \|_{L^2(\Man_h)}
\\
& \lesssim \| e_r  \|_{W^{1, \infty}(\Man_h)} \| e_r \|_{H^1(\Man_h)}
\bigl( \normt{e_v}_h + h^{-1}\normt{e_r}_h + h^k \bigr)
\\
& \lesssim (h^{-4} \normt{e_r}_h )\normt{e_r}_h
\bigl( h\normt{e_v}_h + \normt{e_r}_h + h^{k+1} \bigr)
\qquad \text{(by the inverse estimate \eqref{eq:inv})}
\\
&\lesssim \normt{e_r}_h^2 + h^2\normt{ e_v}_h^2 + h^{2k+2},
\qquad\qquad\qquad\qquad\qquad \text{(by estimate \eqref{eq:t*})}
\end{aligned}
\]
where, in the last step, we use \eqref{eq:t*} to bound \(h^{-4}\normt{e_r}_h \le h^{\varepsilon}\) .
Similarly, using the inverse estimate \eqref{eq:inv} and \eqref{eq:t*}, we obtain the estimate for \(K_2\):
\[
K_2 \leq \| e_r \|_{W^{1,\infty}(\Man_h)} \| e_v \|_{H^1(\Man_h)} \normt{e_r}_h
\lesssim (h^{-4}\normt{e_r}_h) (h\normt{e_v}_h) \normt{e_r}_h
\lesssim h^2\normt{ e_v }_h^2 + \normt{ e_r }_h^2,
\]
where, again in the last step, we use \eqref{eq:t*} to bound \(h^{-4}\normt{e_r}_h \le h^{\varepsilon}\).
For the estimate of $K_3$, we have 
\[ 
    K_3 \leq \normt{ d_v }_{h,*}^2 + \normt{e_r}_h^2.
 \]
Combining the estimates of $K_1, K_2, K_3$ and applying Lemma \ref{lm:esti_ev}, we obtain 
\[ \begin{aligned}
  \frac{\mathrm{d}}{\mathrm{d}t}\|\mathrm{d} e_{r} \odot \mathrm{d} r_h^*\|_{L^2(\Man_h)}^2  
  &\lesssim    \normt{e_r}_h^2 + h^2\normt{ e_v}_h^2 + \normt{ d_v }_{h,*}^2 + h^{2k+2}  \\ 
  &\lesssim   
   \normt{e_r}_h^2 + \normt{ d_v }_{h,*}^2 + \|d_\lambda \|_{L^2(\Man_h)}^2 + h^{2k+2}.
\end{aligned} \]
That concludes the proof of Proposition \ref{prop:esti_drder} by integration from $0$ to $t$.
\end{proof}

We now turn to the estimate of $\| \Pi_{\RM} e_r\|_{L^2(\Man_h)}$ by considering the error equation in \eqref{eq:err-b}, which can be rewritten as follows:  
\begin{equation}\label{eq:err-b-re} 
\begin{aligned}
    ( \partial_t (\Pi_{\RM} e_r), \mu_h^* )_{\Man_h} = (\partial_t (\Pi_{\RM} e_r) - \partial_t e_r, \mu_h^* )_{\Man_h} 
    - (\partial_t r_h, \mu_h^{(1)} \times e_r)_{\Man_h} - (d_{\lambda}, \mu_h^*)_{\Man_h},
\end{aligned}
\end{equation} 
for all $ \mu_h^* = \mu_h^{(1)} \times r_h^* + \mu_h^{(2)} \in \RM[r_h^*] $ with $\mu_h^{(1)}, \mu_h^{(2)} \in \mathbb{R}^3$.

\begin{proposition}\label{prop:esti_eRM}
    Under the assumptions of Theorem \ref{thm:err}, there exists a constant $h_0 > 0$ such that the following stability result holds uniformly for $0<h\leq h_0$ and $t\in[0,t^{*}]${\rm:}
    \begin{equation}\label{eq:esti_eRM} 
      \|  (\Pi_{\RM} e_r)(t)\|_{L^2(\Man_h)}^2  \lesssim \int_0^t \big(\|e_r(s)\|_{L^2(\Man_h)}^2 + \|d_\lambda(s)\|_{L^2(\Man_h)}^2\big)\mathrm{d}s.
    \end{equation} 
\end{proposition}
\begin{proof}
    Taking $\mu_h^* =  \Pi_{\RM} e_r =   e_{r, \RM}^{(1)} \times r_h^* +  e_{r, \RM}^{(2)}$ in \eqref{eq:err-b-re} yields 
\[ \begin{aligned}
        \frac{1}{2} \frac{\mathrm{d}}{\mathrm{d}t} 
    \| \Pi_{\RM} e_r\|_{L^2(\Man_h)}^2 &= (\Pi_{\RM}(\partial_t (\Pi_{\RM} e_r) - \partial_t e_r), \Pi_{\RM}e_r )_{\Man_h}
    - (\partial_t r_h,  e_{r,\RM}^{(1)} \times e_r)_{\Man_h} \\
    &\quad - (d_{\lambda},\Pi_{\RM}e_r )_{\Man_h}.
\end{aligned} \]
Invoking the boundedness \eqref{eq:boundedness}, we obtain 
\[   \frac{\mathrm{d}}{\mathrm{d}t} 
    \| \Pi_{\RM} e_r\|_{L^2(\Man_h)}^2
    \lesssim \| e_r \|_{L^2(\Man_h)}^2 + \|d_\lambda\|_{L^2(\Man_h)}^2 
    + \|\Pi_{\RM}(\partial_t (\Pi_{\RM} e_r) -\partial_t e_r)\|_{L^2(\Man_h)}^2.\]
Now, it remains to estimate the term $\| \Pi_{\RM}(\partial_t (\Pi_{\RM} e_r) - \partial_t e_r)\|_{L^2(\Man_h)}^2$. 
For this, we recall the definition of \( \Pi_{\RM} \) in \eqref{eq:Pi_RM} to derive
\[ ( \Pi_{\RM} e_r - e_r, \alpha \times r_h^* + \beta )_{\Man_h} = 0, \quad \forall \alpha,\beta \in \mathbb{R}^3.  \]    
Taking the time derivative gives 
\[ ( \partial_t(\Pi_{\mathrm{RM}} e_r) - \partial_t
e_r,\alpha \times r_h^* + \beta)_{\Man_h}
= - ( \Pi_{\mathrm{RM}} e_r - e_r,\alpha \times \partial_t r_h^* )_{\Man_h} , \quad \forall \alpha,\beta \in \mathbb{R}^3.  \]
Next, the norm equivalence in the finite dimensional space $\RM[r_h^*]$ leads to 
\[ 
\begin{aligned}
    \| \Pi_{\RM}(\partial_t (\Pi_{\RM} e_r) - \partial_te_r)\|_{L^2(\Man_h)} & \eqsim
\sup_{\alpha, \beta \in \mathbb{R}^3} \frac{(\Pi_{\RM}(\partial_t (\Pi_{\RM} e_r) -\partial_t e_r), \alpha \times r_h^* + \beta)_{\Man_h}}{|\alpha| + |\beta|}  \\ 
& = 
\sup_{\alpha, \beta \in \mathbb{R}^3} \frac{(\partial_t (\Pi_{\RM} e_r) -\partial_t e_r, \alpha \times r_h^* + \beta)_{\Man_h}}{|\alpha| + |\beta|}  \\ 
& = 
\sup_{\alpha, \beta \in \mathbb{R}^3} \frac{(  e_r - \Pi_{\mathrm{RM}}e_r,\ \alpha \times \partial_t r_h^* )_{\Man_h}}{|\alpha| + |\beta|} \\
&\le \sup_{\alpha \in \mathbb{R}^3} \frac{(  e_r - \Pi_{\mathrm{RM}}e_r,\ \alpha \times \partial_t r_h^* )_{\Man_h}}{|\alpha| }  \\
&
\lesssim \|e_r\|_{L^2(\Man_h)}.
\end{aligned}
\]
This concludes the proof of \eqref{eq:esti_eRM}. 
\end{proof}

\subsection{Error estimate}
We are now ready to prove Theorem \ref{thm:err} by combining the stability estimates from Proposition \ref{prop:esti_drder} and Proposition \ref{prop:esti_eRM}. 

\begin{proof}[Proof of Theorem \ref{thm:err}]  
Combining \eqref{eq:esti_drder}, \eqref{eq:esti_eRM} and \eqref{eq:graphnorm_esti}, we have, for $t \in [0,t^*]$ 
\[ \begin{aligned}
    \normt{e_r(t)}^2_h & \lesssim 
     \|  (\Pi_{\RM} e_r)(t)\|_{L^2(\Man_h)}^2 + \|(\mathrm{d} e_{r} \odot \mathrm{d} r_h^*)(t)\|_{L^2(\Man_h)}^2 \\ 
    & \lesssim \int_0^t \big( 
\normt{e_r(s)}^2_h + \| d_\lambda(s) \|_{L^2(\Man_h)}^2 + \normt{d_v(s)}_{h,*}^2 \big)\mathrm{d}s + h^{2k+2}. 
\end{aligned} \]
By Gr\"onwall's inequality and Lemma \ref{lm:def-esti} (Defect estimates), we obtain
\begin{equation}\label{eq:error-estimate} 
\normt{e_r(t)}_h \lesssim 
 \Big(\int_0^t \big( 
 \| d_\lambda(s) \|_{L^2(\Man_h)}^2 + \normt{d_v(s)}_{h,*}^2 \big)\mathrm{d}s
+h^{2k+2}\Big)^{\frac{1}{2}} \lesssim h^{k}, \quad t \in [0,t^*].
\end{equation} 
Moreover, by applying the norm equivalence~\eqref{eq:norm-equi}, 
the inverse estimate~\eqref{eq:inv}, 
and the approximation property of $r_h^*$, 
we obtain that
\[
\begin{aligned}
\| r_h^{(\ell)}(t) - r(t) \|_{W^{1,\infty}(\Man)}
&\lesssim 
h^{-2}\| e_r^{(\ell)}(t) \|_{L^2(\Man)}
+ \| (r_h^*)^{(\ell)}(t) - r(t) \|_{W^{1,\infty}(\Man)} \\
&\lesssim 
h^{-2} \| e_r(t) \|_{L^2(\Man_h)}
+ h^{k} \\
&\le 
h^{-2}\normt{e_r(t)}_h
+ h^{k}
\lesssim h^{k-2},
\qquad t \in [0,t^*].
\end{aligned}
\]
Since \(k \ge 5\), for sufficiently small \(h\), the above estimates imply that, for all \(t \in [0, t^*]\),
\[ \normt{e_r(t)}_h \lesssim h^k < h^{4 + \varepsilon}, \quad \text{and} \quad \|r_h^{(\ell)}(t) - r(t)\|_{W^{1, \infty}(\Man)} \lesssim h^{k-2} < h^{1+\varepsilon}. \]
This implies that the condition \eqref{eq:dis_Korn_cond} in Lemma~\ref{lm:dis_Korn} is satisfied by \(r_h(t)\)
for all \(t \in [0,t^*]\).
By the discussion in Section~\ref{subsec:wellpose_num}, this ensures that the solution of the finite element scheme
\eqref{eq:num} can be extended continuously to an interval strictly larger than \([0,t^*]\). 
Hence, there exists \(\delta > 0\) such that the finite element solution of \eqref{eq:num} still exists and satisfies the estimate \eqref{eq:t*} on a larger interval \([0, t^* + \delta]\).
Since \(t^* \in (0, T]\) is the supremum of times for which \eqref{eq:t*} holds on \([0, t^*]\), it follows that \(t^* = T\), and therefore $ r_h \in C^1([0,T]; V_h^3) $ and the error estimate \eqref{eq:error-estimate} is valid for all \(t \in [0, T]\).

Finally, invoking \eqref{eq:norm-equi}, the boundedness \eqref{eq:boundedness} together with the approximation property of $r_h^*$, we deduce that 
\[ 
\begin{aligned}
    & \quad \normt{r(t) - r_h^{(\ell)}(t) }_{\Man} \\
    &=\big(\| r - r_h^{(\ell)} \|_{L^2(\Man)}^2 + 
    \| \mathrm{d} r \odot \mathrm{d} (r - r_h^{(\ell)}) \|_{L^2(\Man)}^2\big)^{\frac{1}{2}} \\
    & \eqsim \| r^{(-\ell)} - r_h \|_{L^2(\Man_h)} 
    + \| \mathrm{d} r^{(-\ell)} \odot \mathrm{d} (r^{(-\ell)} - r_h) \|_{L^2(\Man_h)} \\ 
    & \leq\normt{e_r}_h + 
    \| r^{(-\ell)} - r_h^* \|_{L^2(\Man_h)}  + 
    \| \mathrm{d} r^{(-\ell)} \odot \mathrm{d} (r^{(-\ell)} - r_h) -\mathrm{d} r_h^* \odot \mathrm{d} (r_h^* - r_h)  \|_{L^2(\Man_h)} \\ 
    & \lesssim h^k,
\end{aligned}
\]
for all $t \in [0,T]$.
That concludes the proof of Theorem \ref{thm:err}. 
\end{proof}

\section{Numerical experiments}\label{sec:numerical}

In this section, we present numerical examples to illustrate the convergence of the proposed scheme in \eqref{eq:num}, as well as the simulation of isometric embeddings of evolving metrics. 
In all the numerical examples, a 3-step linearly semi-implicit backward differentiation formula (BDF) is used for time discretization, with a sufficiently small stepsize to guarantee that the errors from time discretization are negligibly small compared with the errors from spatial discretization. At each time level, only a linear system for  
$v_h$ is solved, with $r_h$ being expressed in terms of $v_h$ using the BDF method for \eqref{eq:num-v}. 

The numerical experiments presented below are performed using the open-source finite element library NGSolve~\cite{schoberl2014c++}.

\begin{example}[Convergence rates for isometric embedding]
    \upshape
In the first example, we test the convergence rates of our numerical scheme for computing the isometric embedding of an evolving metric. Let the manifold \( \Man \) be the ellipsoid defined by
\[
    \Man = \left\{ \left( \begin{array}{c}
        \frac{1}{2}\sin \varphi \cos \theta \\[5pt]
        \frac{1}{2} \sin \varphi \sin \theta \\[5pt]
        \cos \varphi
    \end{array} \right) : \quad \theta \in [0, 2\pi), \quad \varphi \in [0, \pi] \right\},
\]
and let \( g(t) \) be the evolving metric on \( \Man \), with \( t \in [0,1] \), defined as
\[
    g(t) = \mathrm{d} r(t) \odot \mathrm{d} r(t),
\]
where the exact smooth embedding \( r(t): \Man \rightarrow \mathbb{R}^3 \) is given by
\begin{equation}\label{eq:exa_emb}
    p = \left( \begin{array}{c}
        x \\[1pt]
        y \\[1pt]
        z
    \end{array} \right) \in \Man \mapsto
    r(t, p) = \left( \begin{array}{c}
        (1 - t)x + \frac{t}{2} x \\[1pt]
        (1 - t)y + \frac{t}{2} y \\[1pt]
        (1 - t)z + \frac{t}{3} z
    \end{array} \right), \quad t \in [0,1].
\end{equation}
It is easy to see that \(g(t)\) corresponds to a smooth deformation between ellipsoidal metrics.
In particular, the Gaussian curvature of \(g(t)\) remains strictly positive for all \(t\in[0,1]\).
By direct calculation, the embedding \(r(t)\) in \eqref{eq:exa_emb} yields an exact solution of \eqref{eq:vel}.

The errors of the numerical solutions are measured in the discrete graph norm \eqref{eq:graph_norm} of the error function \( e_r = r_h - r_h^* \) at time \( T = 0.1 \), 
and these are presented in Figure \ref{fig:rate} for several different mesh sizes \( h = 0.7, 0.6, 0.5, 0.4, 0.3, 0.25 \). 
The numerical results in Figure \ref{fig:rate} show that the errors of the numerical solutions scale approximately as \( O(h^{k}) \), with \( k = 5, 6 \), 
which is consistent with the theoretical results established in Theorems~\ref{thm:err}.

    \begin{figure}[htbp]
        \centering
        \begin{subfigure}{0.47\textwidth}
            \includegraphics[width=\textwidth]{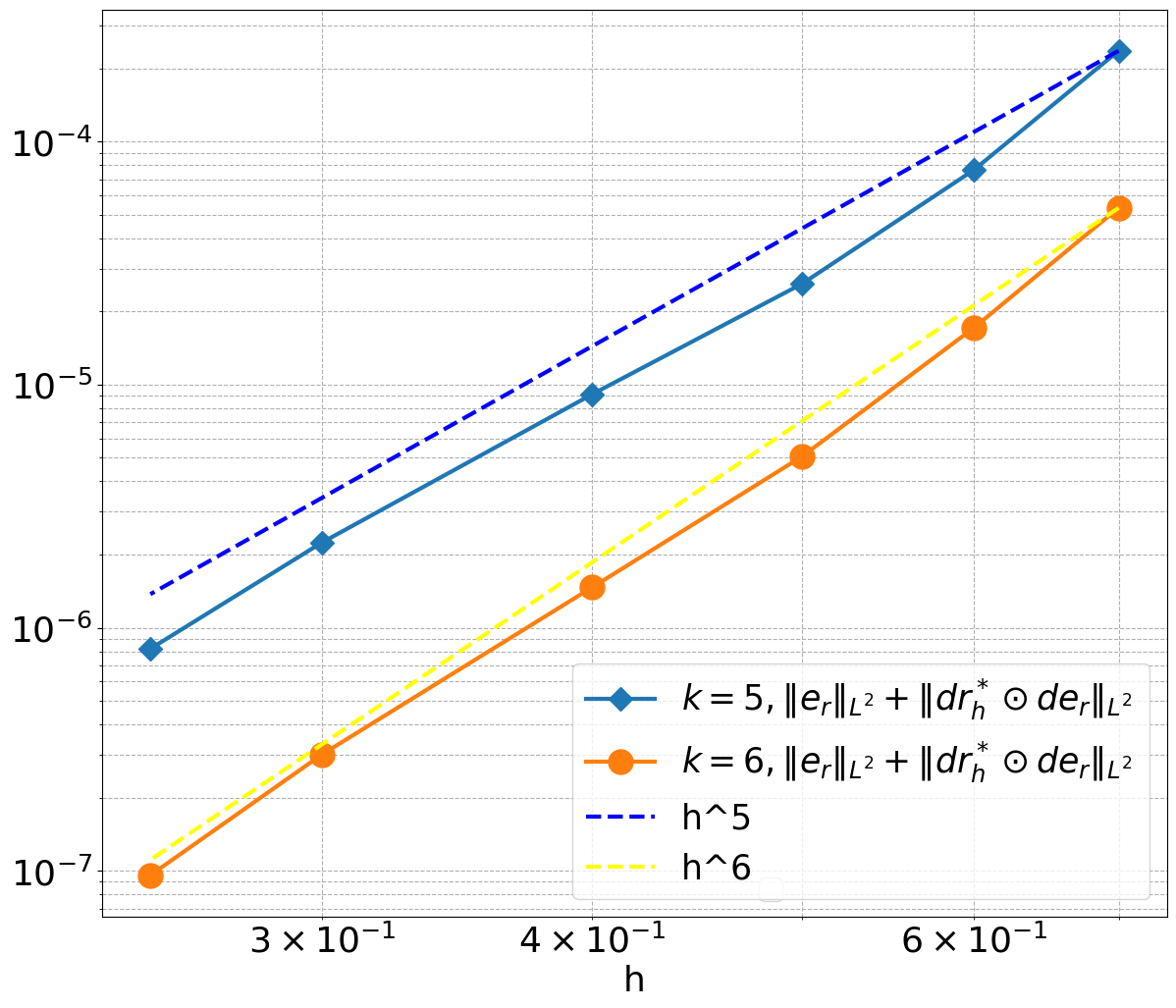}
        \end{subfigure}
\caption{Errors and convergence rate of the numerical scheme in \eqref{eq:num}}
            \label{fig:rate}
    \end{figure}

\end{example}
\begin{example}[Isometric embedding of a revolution metric]

In this example, we simulate the isometric embedding of revolution metrics on the sphere \( S^2 \). 
Consider a one-dimensional curve lying in the \( x \)-\( z \) plane, given by
\[
    \gamma(t) = \{ (x(s,t), 0, z(s,t)) : s \in [0,\pi] \}.
\]
By revolving this curve around the \( z \)-axis through an angle of \( 2\pi \), we obtain a surface of revolution:
\[
    \Gamma(t) = \{ r(t,s,\theta) = (x(s,t)\cos\theta,\, x(s,t)\sin\theta,\, z(s,t)) : 
    s \in [0,\pi],\; \theta \in [0,2\pi) \}.
\]
The tangent vectors of \( \Gamma(t) \) are given by
\[
\begin{aligned}
    \partial_s r(t,s,\theta) &= (\partial_s x(s,t)\cos\theta,\, \partial_s x(s,t)\sin\theta,\, \partial_s z(s,t)), \\
    \partial_\theta r(t,s,\theta) &= (-x(s,t)\sin\theta,\, x(s,t)\cos\theta,\, 0).
\end{aligned}
\]
If we use \( (s, \theta) \in [0, \pi] \times [0, 2\pi) \) to denote the polar coordinates on \( S^2 \), the \emph{revolution metric} on \( S^2 \) is given by
\begin{equation}\label{eq:rev_met}
\begin{aligned}
    g(t,s,\theta) 
    &= (\partial_s r \cdot \partial_s r)\, \mathrm{d}s^2 
     + (\partial_\theta r \cdot \partial_\theta r)\, \mathrm{d}\theta^2 \\
    &= \big((\partial_s x(s,t))^2 + (\partial_s z(s,t))^2\big)\, \mathrm{d}s^2
     + (x(s,t))^2\, \mathrm{d}\theta^2.
\end{aligned}
\end{equation}
It is easy to see that the revolution metric \( g(t) \) has strictly positive Gaussian curvature, provided that the corresponding generating curve \( \gamma(t) \) is strictly convex.

In this example, we choose the generating curve defined by
\[
    x(s,t) = \sin(s)\left( (1 - 0.32t) + 0.48t( \cos^2(s) - 1 )^2 \right),
    \qquad
    z(s,t) = \cos(s),
    \qquad s \in [0,\pi].
\]
Note that when \( t = 0 \), the surface \( \Gamma(0) \) coincides with the unit sphere \( S^2 \), and when \( t = 1 \), the surface \( \Gamma(1) \) is a surface of revolution, with the Gaussian curvature varying from \( \kappa_{\rm min} = 0.055 \) to \( \kappa_{\rm max} = 4.68 \).
We discretize the isometric embedding using the scheme \eqref{eq:num}, with finite elements of degree \( k = 5 \) and mesh size \( h = 0.35 \).
The numerical simulation of the isometric embedding corresponding to the revolution metric \eqref{eq:rev_met} is shown in Figure \ref{fig:rev_evolution}.
As \( t \) varies from \( 0 \) to \( 1 \), the surface gradually deforms from the sphere to the corresponding surface of revolution, accurately reflecting the geometric change.

\begin{figure}[htbp]
    \centering
    \begin{subfigure}[t]{0.3\textwidth}
        \centering
        \includegraphics[width=\textwidth]{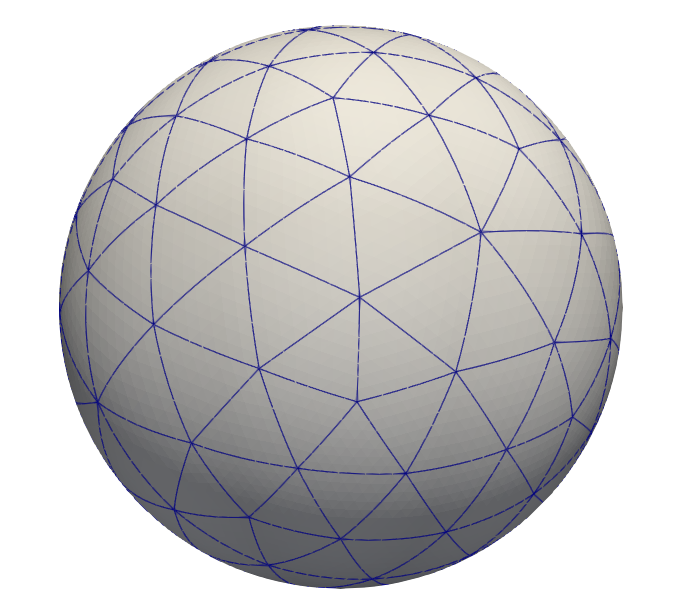}
        \caption{$t = 0$}
    \end{subfigure}

    \vspace{0.8em} 

    \begin{subfigure}[t]{0.3\textwidth}
        \centering
        \includegraphics[width=\textwidth]{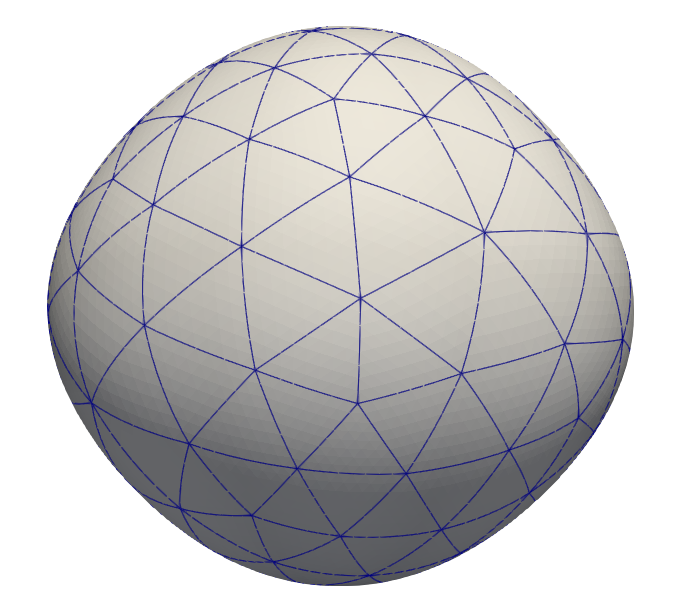}
        \caption{$t = 0.25$}
    \end{subfigure}
    \hfill
    \begin{subfigure}[t]{0.3\textwidth}
        \centering
        \includegraphics[width=\textwidth]{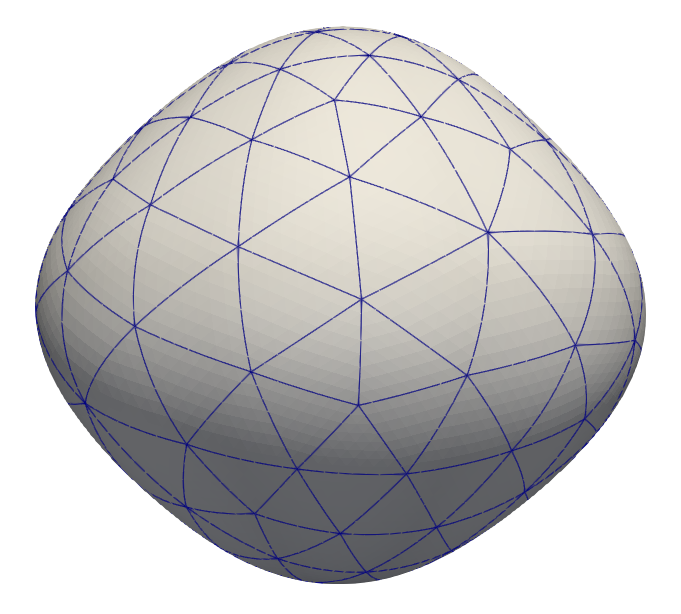}
        \caption{$t = 0.5$}
    \end{subfigure}
    \hfill
    \begin{subfigure}[t]{0.3\textwidth}
        \centering
        \includegraphics[width=\textwidth]{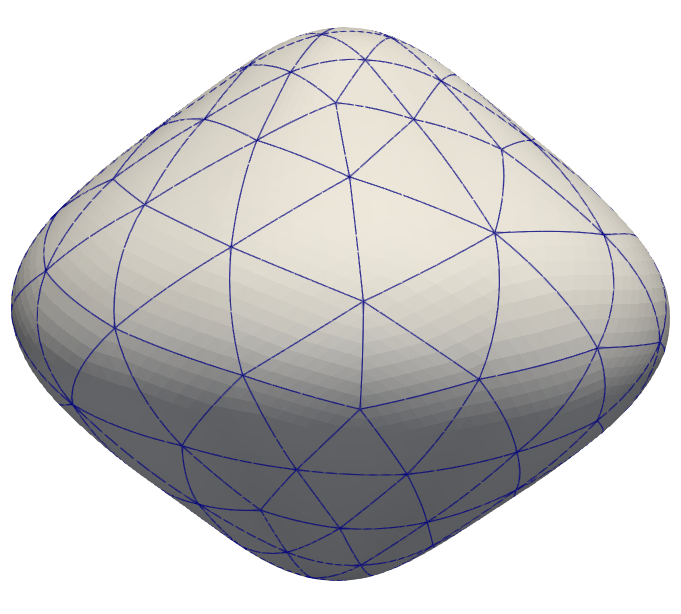}
        \caption{$t = 1$}
    \end{subfigure}

    \caption{Numerical simulation of the isometric embedding of the revolution metric \eqref{eq:rev_met} at different time steps.}
    \label{fig:rev_evolution}
\end{figure}

\end{example}

\begin{example}[Visualization of the normalized Ricci flow]
As an application of the proposed numerical scheme for isometric embedding, in this example, we investigate the visualization of the \emph{normalized Ricci flow}. 
Let \( g(t) \) be a family of metrics satisfying
\begin{equation}\label{eq:Ricci-flow}
    \partial_t g(t) = 2 \big( \bar{\kappa} - \kappa(t) \big) g(t),
\end{equation}
where \( \kappa(t) = \kappa(g(t)) \) denotes the Gaussian curvature of \( g(t) \), and 
\( \bar{\kappa} \) is the average of the initial curvature. 
Equation~\eqref{eq:Ricci-flow} is known as the normalized Ricci flow. 
In this example, we take the solution \( g(t) \) of \eqref{eq:Ricci-flow} as the prescribed right-hand side in \eqref{eq:vel}. 
The corresponding isometric embedding 
\( r(t): \Man \to \mathbb{R}^3 \) then provides a geometric visualization in \( \mathbb{R}^3 \) 
of the evolution governed by the Ricci flow \eqref{eq:Ricci-flow}. 
Note that the normalized Ricci flow preserves positive Gaussian curvature. 
Hence, if the initial metric satisfies \( \kappa(0) > 0 \), then \( g(t) \) remains of positive curvature for all \( t \). 
Another important geometric property of the normalized Ricci flow is that 
\( g(t) \) converges to a metric of constant curvature as \( t \to \infty \). 
We will investigate this behavior in the numerical simulation presented below.

In our numerical simulation, we take the following initial surface:
\begin{equation}
\begin{aligned}
    \Man &= \left( \begin{array}{c}
        (0.7\sin (\varphi)+0.1\sin(2\varphi)) \cos \theta \\[5pt]
        (0.7\sin (\varphi)+0.1\sin(2\varphi)) \sin \theta \\[5pt]
        0.5\cos (\varphi)
    \end{array} \right),
    \quad &&\theta \in [0, 2\pi), \quad \varphi \in [0, \pi]. 
\end{aligned}
\end{equation}
The initial metric \( g(0) \) is given by its induced metric.
We employ the finite element methods developed in \cite{gao2025ricci,gawlik2019finite} to discretize the Ricci flow \eqref{eq:Ricci-flow}. In these methods, Regge elements are used for the discretization of the metric \( g(t) \), while scalar Lagrange elements are employed for the discretization of the Gaussian curvature \( \kappa(t) \). The isometric embedding is then approximated using the scheme \eqref{eq:num}. All finite element spaces are taken to be of polynomial degree \( 5 \) with mesh size \( h = 0.2 \). The resulting numerical simulation of the isometric embedding is shown in Figure \ref{fig:ricci-flow}. It can be observed that the surface gradually evolves into a round sphere, corresponding to a metric of constant curvature, accurately capturing the geometric evolution of the Ricci flow.

\begin{figure}[htbp]
    \centering
    \begin{subfigure}[b]{0.31\textwidth}
        \includegraphics[width=\textwidth]{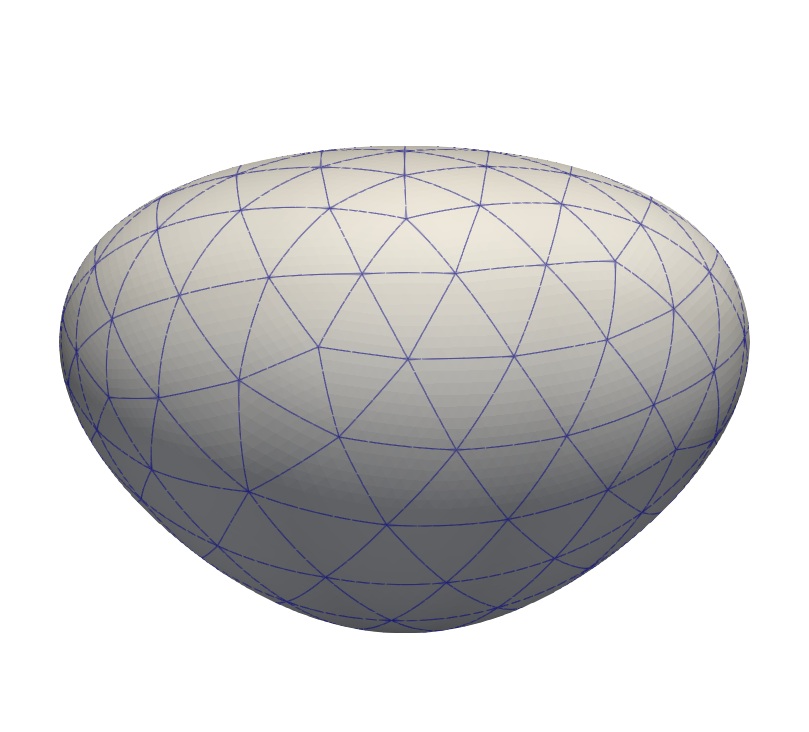}
        \caption{$t = 0$}
    \end{subfigure}
    \hfill
    \begin{subfigure}[b]{0.31\textwidth}
        \includegraphics[width=\textwidth]{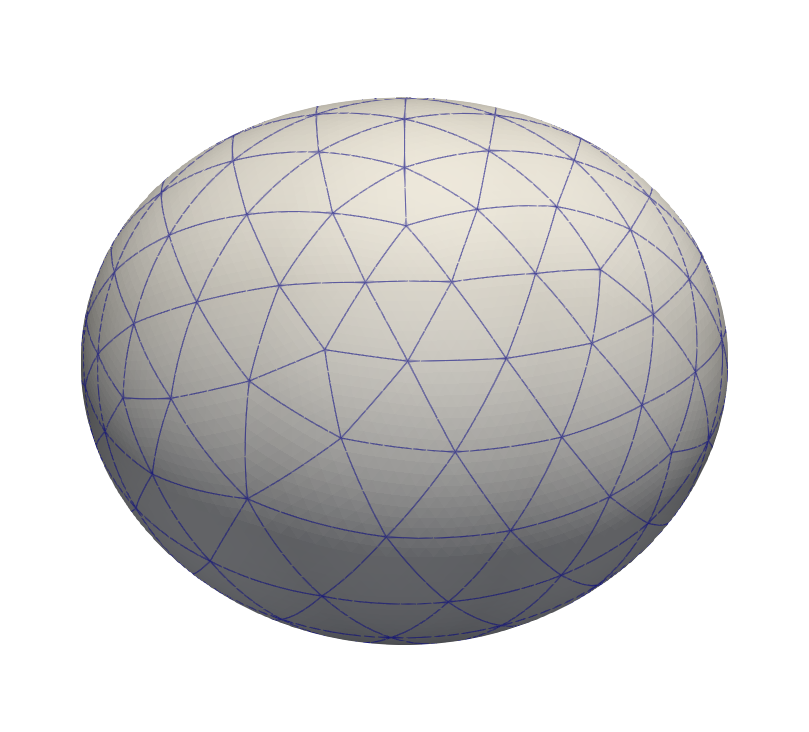}
        \caption{$t = 0.06$}
    \end{subfigure}
    \hfill
    \begin{subfigure}[b]{0.31\textwidth}
        \includegraphics[width=\textwidth]{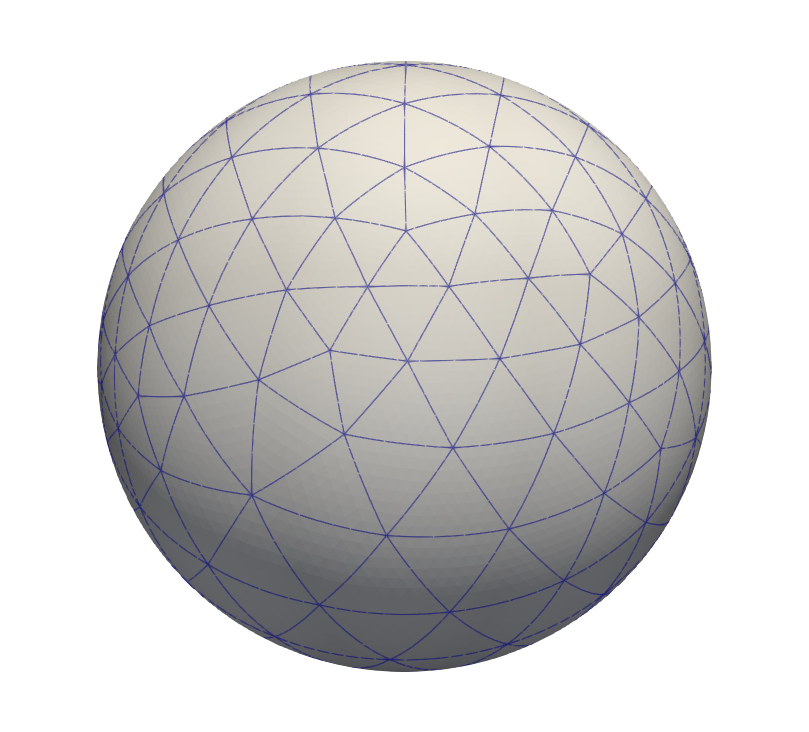}
        \caption{$t = 0.4$}
    \end{subfigure}
    \caption{Visualization of the normalized Ricci flow via isometric embeddings. 
    The surface evolves toward a round sphere corresponding to a constant-curvature metric.}
    \label{fig:ricci-flow}
\end{figure}

\end{example}

\section{Conclusion}  
We have proposed a convergent finite element method for approximating the isometric embedding of a two-dimensional Riemannian manifold with positive Gaussian curvature into \( \mathbb{R}^3 \).
We begin by transforming the problem of embedding a fixed Riemannian metric into a series of dynamic embedding problems, which are subsequently solved by the linear PDEs governing the velocity.
We introduce a new variational formulation to determine the velocity and address the nontrivial kernel of the velocity equations.
This formulation also naturally applies to the visualization of intrinsic curvature flows.
We analyze the well-posedness of the new variational formulation by establishing a Korn inequality on the manifold and discretize it using high-order finite element spaces and tensor-valued finite element spaces.
A discrete version of the Korn inequality is proved, ensuring the well-posedness of the proposed numerical discretization.
Convergence and error estimates are established for polynomial degrees \( k \geq 5 \).
The numerical analysis framework presented in this work represents the first systematic study of the finite element discretization of isometric embeddings and provides a fundamental approach for investigating related problems in the future.

\section*{Acknowledgments}
This work was partially supported by the National Natural Science Foundation of China (Project No. 12525111), and the European Union-Hong Kong Research Cooperation Co-funding Mechanism through the Research Grants Council of Hong Kong (Project No. E-PolyU502/24). 
This work was also supported in part by a Royal Society University Research Fellowship (URF$\backslash$R1$\backslash$221398) and the European Union (ERC, GeoFEM, 101164551).  Views and opinions expressed are however those of the authors only and do not necessarily reflect those of the European Union or the European Research Council. Neither the European Union nor the granting authority can be held responsible for them. 


\bibliographystyle{abbrv}
\bibliography{IsoEmb}

        
\newpage
\appendix
\renewcommand{\thesection}{Appendix}
\section*{More details in the stability estimates} \label{appendix}
\setcounter{section}{0}
\renewcommand{\theequation}{A.\arabic{equation}}
\renewcommand{\thesection}{Appendix \Alph{section}}
\setcounter{equation}{0}

\section{Existence and uniqueness of the continuous formulation}\label{appendix:Wely-emb}

Since the Gaussian curvature of $g(t)$ remains strictly positive for all $t \in [0,T]$, 
it follows from Weyl's embedding theorem \cite[Theorem 9.0.1]{han2006isometric} 
that there exists a smooth embedding 
$\hat r(t): \Man \rightarrow \mathbb{R}^3$, depending smoothly on $t \in [0,T]$, 
which satisfies the isometric embedding equation \eqref{eq:isoemb} for all $t \in [0,T]$. 
By shifting $\hat r(t)$ by a constant vector, we may assume that 
$\int_{\Man} \hat r(t)\, \Vol_{\Man} = 0$. 
Consequently, differentiating this relation with respect to $t$ yields 
$\int_{\Man} \partial_t \hat r(t)\, \Vol_{\Man} = 0$. 
Moreover, differentiating \eqref{eq:isoemb} with respect to $t$ shows that 
$\partial_t \hat r$ satisfies \eqref{eq:pt_isoemb}. 
Next, we will modify $\hat r(t)$ so that it also satisfies \eqref{eq:ptr_orth_RM}.

Firstly, for any orthogonal matrix $R(t) \in \mathrm{O}(3)$ that depends smoothly on $t$, we define a new embedding as
\begin{equation}\label{eq:wely-emb_R} 
    r(t,\cdot) = R(t)\hat r(t,\cdot) : \Man \to \mathbb{R}^3.
\end{equation} 
Then, by direct computation, $r(t,\cdot)$ satisfies \eqref{eq:isoemb}, and its time derivative $\partial_t r$ also satisfies \eqref{eq:pt_isoemb}.
To ensure that \eqref{eq:ptr_orth_RM} also holds, it suffices to find 
$R(t) \in \mathrm{O}(3)$ such that
\[
    \int_{\Man} \partial_t r \,\Vol_{\Man} = 0,
    \qquad\text{and}\qquad
    \int_{\Man} (\partial_t r \times r)\,\Vol_{\Man} = 0.
\]
Indeed, for the first equality we have 
    \[ \int_{\Man}\partial_t r(t)\,\Vol_{\Man} 
    = R(t) \int_{\Man}  \partial_t \hat r(t)\,\Vol_{\Man} + 
     \partial_t R(t) \int_{\Man}  \hat r(t)\,\Vol_{\Man} = 0. \]
On the other hand, we shall determine $R(t) \in \mathrm{O}(3)$ appropriately in order to enforce the second equality. Now, substituting expression \eqref{eq:wely-emb_R}, we obtain 
\[ \begin{aligned}
    \int_{\Man} (\partial_t r \times r)\,\Vol_{\Man} &= 
    \int_{\Man} \big( (\partial_t R \hat r) \times (R \hat r) + 
    (R \partial_t \hat r) \times (R \hat r)
    \big) \Vol_{\Man} \\ 
    &= R \int_{\Man}((R^{\top}\partial_t R \hat{r} ) \times  \hat{r} + 
    \partial_t \hat r \times \hat{r}    
    \big) \Vol_{\Man}.  
\end{aligned} \]
Since $R(t) \in \mathrm{O}(3)$ is an orthogonal matrix, it follows that 
$R^{\top}(t)\partial_t R(t) \in \mathbb{R}^{3\times 3}$ is skew-symmetric, and therefore 
there exists a vector $\eta(t) = (\eta_1, \eta_2, \eta_3)^{\top} \in \mathbb{R}^3$ such that
\[
    R^{\top}\partial_t R = \mathrm{skw}(\eta), 
    \quad \text{where} \quad
    \mathrm{skw}(\eta) \coloneqq
    \begin{pmatrix}
        0 & -\eta_3 & \eta_2 \\
        \eta_3 & 0 & -\eta_1 \\
        -\eta_2 & \eta_1 & 0
    \end{pmatrix}.
\]
Moreover, a direct computation shows that the following identity holds for all 
$\alpha \in \mathbb{R}^3$:
\begin{equation}\label{eq:skew} 
    (R^{\top}\partial_t R)\alpha = \eta \times \alpha.
\end{equation} 
Substituting this, we obtain
\[
\begin{aligned}
\int_{\Man}((R^{\top}\partial_t R \hat{r} ) \times  \hat{r} + 
    \partial_t \hat r \times \hat{r}    
    \big) \Vol_{\Man}
&= \int_{\Man}\!\big((\eta\times\hat r) \times \hat r +\partial_t\hat r \times \hat r \big)\,\Vol_{\Man} \\
&= \Big(\int_{\Man}\!(\hat r\hat r^{\top} - |\hat r|^2 I)\,\Vol_{\Man}\Big)\eta
\;+\; \int_{\Man}\!\partial_t\hat r \times \hat r\,\Vol_{\Man}.
\end{aligned}
\]
Note that 
\(
    \int_{\Man} (|\hat r|^2 I - \hat r \hat r^{\top})\,\Vol_{\Man} 
    \in \mathbb{R}^{3\times 3}
\)
is indeed an invertible matrix. This can be demonstrated by considering any 
$\alpha \in \mathbb{R}^3$, for which we have
\[
    \alpha^{\top}
    \Big( \int_{\Man} (|\hat r|^2 I - \hat r \hat r^{\top})\,\Vol_{\Man} \Big)
    \alpha
    =
    \int_{\Man} \big( |\alpha|^2 |\hat r|^2 - (\hat r \cdot \alpha)^2 \big)\,\Vol_{\Man}
    \geq 0.
\]
Here we have used the Cauchy--Schwarz inequality, and the equality can occur only if
\[
    |\alpha|^2 |\hat r(t,p)|^2 = (\hat r(t,p) \cdot \alpha)^2,
    \qquad \forall p \in \Man.
\]
This in fact implies $\alpha = 0$, since 
$\hat r : \Man \to \mathbb{R}^3$ is an embedding.
Consequently, we set
\[
    \eta 
    = 
    \Big(
        \int_{\Man} (|\hat r|^2 I - \hat r \hat r^{\top})\,\Vol_{\Man}
    \Big)^{-1}
    \int_{\Man} (\partial_t \hat r \times \hat r)\,\Vol_{\Man},
\]
and we determine $R(t) \in \mathrm{O}(3)$ by solving
\[
    \partial_t R = R\,\mathrm{skw}(\eta),
    \qquad R(0) = I_{3 \times 3},
\]
which is a standard ODE on the Lie group $\mathrm{O}(3)$, and hence admits a unique solution.

For the uniqueness, assume we have two embeddings $r_1(t)$ and $r_2(t)$ 
with the same initial value, and that both satisfy 
\eqref{eq:isoemb}, \eqref{eq:pt_isoemb}, and \eqref{eq:ptr_orth_RM}. 
It follows from the rigidity of smooth, closed, convex surfaces 
\cite[Theorem~8.1.2]{han2006isometric} that there exist 
$G(t)\in\mathrm{O}(3)$ and $c(t)\in\mathbb{R}^3$ such that
\begin{equation}\label{eq:r1r2}
r_1(t) = G(t)\, r_2(t) + c(t),
\quad \text{where } G(0)=I_{3\times 3},\; c(0)=0.
\end{equation}
From \eqref{eq:ptr_orth_RM}, we have
\[
\int_{\Man} r_1(t)\,\Vol_{\Man}=0,
\qquad 
\int_{\Man} r_2(t)\,\Vol_{\Man}=0,
\]
which implies $c(t)\equiv 0$. 
Again, from \eqref{eq:ptr_orth_RM}, we also have 
\[
\int_{\Man} (\partial_t r_1 \times r_1)\,\Vol_{\Man}=0,
\qquad
\int_{\Man} (\partial_t r_2 \times r_2)\,\Vol_{\Man}=0.
\]
Using this relation and substituting \eqref{eq:r1r2}, we obtain
\[
\begin{aligned}
0 &= \int_{\Man} (\partial_t r_1 \times r_1)\,\Vol_{\Man} 
= \int_{\Man} \big( (\partial_t G\, r_2 + G\,\partial_t r_2)\times (G r_2) \big)\,\Vol_{\Man} \\
&= \int_{\Man} G\big((G^\top\partial_t G\, r_2 + \partial_t r_2) \times r_2\big)\,\Vol_{\Man} 
= G \int_{\Man} (G^\top\partial_t G\, r_2 \times r_2)\,\Vol_{\Man}.
\end{aligned}
\]
This leads to
\[
\int_{\Man} (G^\top\partial_t G\, r_2 \times r_2)\,\Vol_{\Man} = 0.
\]
Using the fact that $G^\top\partial_t G$ is skew-symmetric and applying the 
same argument as in \eqref{eq:skew}, we conclude that
\(G^\top\partial_t G = 0\),
which implies $G(t)\equiv I_{3\times 3}$.

\section{Proof of the Korn inequality} \label{appendix:Korn}

Our approach is based on the study of the linearized equation for isometric embedding in \cite[Lemma 9.2.2]{han2006isometric}. In what follows, we denote the negative Sobolev space $H^{-k}(\Man)$ as the dual space of $H^{k}(\Man)$, with the norm
\[ \| w \|_{H^{-k}(\Man)} = \sup_{v \in H^{k}(\Man) \setminus \{0\}} \frac{(w, v)_{\Man}}{\|v\|_{H^k(\Man)}}. \]
An analogous definition applies to differential forms, and integration by parts yields
\begin{equation}\label{eq:esti_H-1} 
\begin{aligned}
    \| \mathrm{d} w \|_{H^{-1}(\Man)} &= \sup_{\mu \in H^1\Lambda^1(\Man)\setminus \{0\}} 
\frac{\int_{\Man} \mathrm{d} w \wedge \mu}{\| \mu \|_{H^1(\Man)} }  \\ 
& = \sup_{\mu \in H^1\Lambda^1(\Man)\setminus \{0\}} \frac{-\int_{\Man}  w \wedge \mathrm{d} \mu}{\| \mu \|_{H^1(\Man)} } 
\lesssim \| w \|_{L^2(\Man)}, \quad \forall w \in H^1(\Man),
\end{aligned}
\end{equation} 
where $H^1\Lambda^1(\Man)$ is the space of 1-forms with $H^1$ coefficients.

\smallskip

We prove the Theorem \ref{lm:Korn} by the following steps. We first assume $v\in (\RM[r])^{\perp}$ is sufficiently smooth.

\noindent{\it Step 1.}
Given any smooth, symmetric, second–order covariant tensor field 
$q = q_{ij}\,\mathrm{d}x^i \mathrm{d}x^j$ on $\Man$, we shall prove that there exists 
a smooth map $\tilde{v}:\Man \rightarrow \mathbb{R}^3$ satisfying 
\begin{equation}\label{eq:drdv} 
    \mathrm{d} r \odot \mathrm{d} \tilde{v} = q,
\qquad\text{and}\qquad
\int_{\Man} \tilde{v}\,\Vol_{\Man} = 0.
\end{equation} 
This is a direct consequence of the arguments in 
\cite[Lemma~9.2.2]{han2006isometric}; we state it here for completeness and 
refer the reader to the proof of \cite[Lemma~9.2.2]{han2006isometric} for further details.

For later use, we denote \( h_{ij} = \partial_{ij} r \cdot n \) to be the second fundamental form. Since the Gaussian curvature is strictly positive, the matrix \( (h_{ij})_{1 \leq i,j \leq 2} \) is symmetric positive definite. We introduce new variables as (see also \cite[(9.2.8-9.2.9)]{han2006isometric})
    \[ \begin{aligned}
        u_i &= n \cdot \partial_i \tilde  v, \quad i = 1,2, \\ 
        w &= \frac{1}{\sqrt{|g|}} 
        (\partial_2 r \cdot \partial_1 \tilde v - \partial_1 r \cdot \partial_2 \tilde  v),
    \end{aligned} \]
    where $|g| = \det (g_{ij})_{1 \leq i,j \leq 2}$. Note that $u= u_i \mathrm{d} x^i = n \cdot \mathrm{d} \tilde  v $ is a globally defined 1-form on $\Man$. 
    As shown in \cite[(9.2.12)]{han2006isometric}, by applying \eqref{eq:drdv}, the new variables \( \{ u_1, u_2, w \} \) determine \( \{ \partial_1 \tilde v, \partial_2 \tilde v \} \) as follows:
    \begin{equation}\label{eq:pv} 
    \begin{aligned}
        \partial_1 \tilde v &=  \frac{1}{2} g^{ij} q_{1j} \partial_i r + 
        \frac{1}{2} w \sqrt{|g|} g^{2i}\partial_i r + u_1 n, 
        \\  
        \partial_2 \tilde v &=  \frac{1}{2} g^{ij} q_{2j} \partial_i r - 
        \frac{1}{2} w \sqrt{|g|} g^{1i}\partial_i r + u_2 n. 
    \end{aligned}
    \end{equation}  
Following \cite[(9.2.17), (9.2.20)]{han2006isometric}, the new variables \( \{ u_1, u_2, w \} \) satisfy a system  
\begin{subequations}\label{eq:u1u2w}  
\begin{align}
\label{eq:u1} 
u_1 &= \frac{\sqrt{|g|}}{2} h^{2i} \partial_i w 
      - \frac{\sqrt{|g|}}{2} h^{2i} \gamma_i, \\
\label{eq:u2}
u_2 &= -\frac{\sqrt{|g|}}{2} h^{1i} \partial_i w 
        + \frac{\sqrt{|g|}}{2} h^{1i} \gamma_i, \\
\label{eq:w}
-\frac{1}{\sqrt{|g|}} \, \partial_i \bigl( \sqrt{|g|} \, h^{ij} \partial_j w \bigr) 
- 2 H w 
&= -\frac{1}{\sqrt{|g|}} \, \partial_i \bigl( \sqrt{|g|} \, h^{ij} \gamma_j \bigr) + T.    
\end{align}
\end{subequations}
Here, $H = \frac{1}{2}g^{ij}h_{ij}$ is the mean curvature, \( \gamma_i \) and \( T \) are given by
\begin{subequations}\label{eq:}  
\begin{align}
\label{eq:ci} 
\gamma_i &= \frac{1}{\sqrt{|g|}}
\big( \partial_1 {q}_{2i} - \partial_2 {q}_{1i}
+ \Gamma^{k}_{2i} {q}_{1k} - \Gamma^{k}_{1i} {q}_{2k} \big), \quad i = 1,2, \\ 
\label{eq:T} 
T &= \frac{1}{\sqrt{|g|}}
\big( h^{i}_{2} {q}_{1i} - h^{i}_{1} {q}_{i2} \big), 
\end{align}
\end{subequations}
where \( \Gamma_{ij}^m = \frac{1}{2} g^{ml} ( \partial_j g_{il} + \partial_i g_{lj} - \partial_l g_{ji} ) \) are the Christoffel symbols, and \( h^i_j = g^{il} h_{lj} \). 
Note that \(\gamma\coloneqq \gamma_i \, \mathrm{d}x^i \) is a globally defined 1-form on \( \Man \). 
We now focus on \eqref{eq:w} and define
\[
\mathcal{L} w \coloneqq -\frac{1}{\sqrt{|g|}} \, \partial_i \bigl( \sqrt{|g|} \, h^{ij} \partial_j w \bigr) - 2 H w.
\]
Since \( (h_{ij})_{1 \leq i,j \leq 2} \) is symmetric positive definite, it follows that \( \mathcal{L} \) is an elliptic operator.
Furthermore, \cite[pp. 173-174]{han2006isometric} shows that the kernel space of \( \mathcal{L} \) is given by
\[
\Ker(\mathcal{L}) = \{ \alpha \cdot n : \alpha \in \mathbb{R}^3 \},
\]
and the right-hand side of \eqref{eq:w} is perpendicular to the kernel space \( \Ker(\mathcal{L}) \).
Therefore, by the theory of second-order elliptic equations, \eqref{eq:w} admits a solution \( w \) with \( w \perp \Ker(\mathcal{L}) \).
With such a \( w \), we can solve for \( u_1 \) and \( u_2 \) in \eqref{eq:u1} and \eqref{eq:u2}.
By integrating \eqref{eq:pv}, we obtain \( \tilde v \) satisfying \eqref{eq:drdv}; moreover, prescribing the value of \( \tilde v \) at one point ensures that \( \int_{\Man} \tilde v \, \Vol_{\Man} = 0 \).

\noindent{\it Step 2.} 
We establish that the solution $\tilde{v} : \Man \rightarrow \mathbb{R}^3$ of \eqref{eq:drdv} 
satisfies the following estimate:
\begin{equation}\label{eq:drdv_esti}
\|\tilde{v}\|_{L^2(\Man)} + \|P\tilde{v}\|_{H^1(\Man)}
\lesssim \|w\|_{L^2(\Man)} + \|q\|_{L^2(\Man)}.
\end{equation}
Indeed, it follows from the product rule 
$\partial_i(P \tilde v) = (\partial_i P) \tilde  v + P (\partial_i\tilde  v)$ that 
\[ \begin{aligned}
    | \partial_i(P \tilde v) | &\lesssim |\tilde v| + |P \partial_i \tilde v| \\ 
    & \lesssim  |\tilde v| + |\partial_1 r \cdot \partial_i \tilde v|
    + |\partial_2 r \cdot \partial_i \tilde v| \\ 
    & \lesssim  |\tilde v| + |w| + \sum_{j,l = 1}^2 |q_{jl}|, 
\end{aligned} \]
where in the last step, we substitute the expression for $\partial_i \tilde v$ from \eqref{eq:pv}. Hence 
\begin{equation}\label{eq:pv_esti} 
\|P \tilde  v\|_{H^1(\Man)} \lesssim 
\| \tilde  v \|_{L^2(\Man)} + \| w \|_{L^2(\Man)} + \|q\|_{L^2(\Man)}.
\end{equation} 
Next, for the estimate of $\|\tilde v\|_{L^2(\Man)}$, recall that $\int_{\Man} \tilde  v \Vol_{\Man} = 0$. Then there exists $\psi_{\tilde v} \in H^2(\Man; \mathbb{R}^3)$ such that 
\[ (\mathrm{d} \psi_{\tilde v}, \mathrm{d} \chi)_{\Man} = (\tilde v, \chi)_\Man, \quad \forall \chi \in H^1(\Man; \mathbb{R}^3). \]
By substituting $\chi = \tilde  v$, we obtain 
\[ (\tilde v, \tilde v)_{\Man} = (\mathrm{d} \psi_{\tilde v}, \mathrm{d} \tilde v)_{\Man}
\lesssim \| \mathrm{d} \tilde v \|_{H^{-1}(\Man)} \|\psi_{\tilde v}\|_{H^2(\Man)}
\lesssim \| \mathrm{d} \tilde v \|_{H^{-1}(\Man)} \| \tilde v \|_{L^2(\Man)}, 
\]
where in the last step, we use the elliptic regularity estimate $ \|\psi_{\tilde v}\|_{H^2(\Man)} \lesssim \| \tilde v \|_{L^2(\Man)}$. 
Then, by substituting the expression for \( \partial_i \tilde v \) in \eqref{eq:pv} along with the equations \eqref{eq:u1} and \eqref{eq:u2}, the formula for \( \gamma_i \) in \eqref{eq:ci}, and the estimate in \eqref{eq:esti_H-1}, we arrive at
\begin{equation}\label{eq:v_esti} 
\begin{aligned}
    \|\tilde  v\|_{L^2(\Man)} &\lesssim \| \mathrm{d} \tilde v \|_{H^{-1}(\Man)} 
\lesssim \| w \|_{H^{-1}(\Man)} + \|q\|_{H^{-1}(\Man)} + \|u\|_{H^{-1}(\Man)} \\ 
& \lesssim \| w \|_{L^2(\Man)} + \|q\|_{L^2(\Man)} + \| \gamma\|_{H^{-1}(\Man)} 
+ \|\mathrm{d} w\|_{H^{-1}(\Man)} \\ 
& \lesssim \| w \|_{L^2(\Man)} + \|q\|_{L^2(\Man)}.
\end{aligned}
\end{equation} 
Combining \eqref{eq:v_esti} and \eqref{eq:pv_esti}, we deduce \eqref{eq:drdv_esti}, as desired.

\noindent{\it Step 3.} We then control $\|w\|_{L^2(\Man)}$ in terms of $\|q\|_{L^2(\Man)}$. This follows from an \(L^2\)-estimate for the elliptic equation \eqref{eq:w}, as shown below.

Since \( w \) is orthogonal to \( \Ker(\mathcal{L}) \), by the theory of elliptic equations, there exists \( \psi_w \in H^2(\Man) \) such that
\[
\mathcal{L} \psi_w = w, \quad \text{and} \quad 
\| \psi_w \|_{H^2(\Man)} \lesssim \| w \|_{L^2(\Man)}.
\]
Substituting \eqref{eq:w} and using that $\mathcal{L}$ is self-adjoint,  we conclude that
\begin{equation}\label{eq:esti_w}
\begin{aligned}
         \int_\Man w^2 \Vol_{g} &= \int_\Man w (\mathcal{L} \psi_w) \Vol_{g}
     = \int_\Man (\mathcal{L} w)  \psi_w\Vol_{g}\\
     &\lesssim \| \mathcal{L} w  \|_{H^{-2}(\Man)} \| \psi_w \|_{H^2(\Man)} \\ 
     & \leq \Big( \| \frac{1}{\sqrt{|g|}} \, \partial_i \bigl( \sqrt{|g|} \, h^{ij} \gamma_j \bigr)\|_{H^{-2}(\Man)} + \| T \|_{H^{-2}(\Man)}
        \Big) \| w \|_{L^2(\Man)} \\
     & \lesssim \Big( \|  \gamma \|_{H^{-1}(\Man)} + \| T \|_{L^2(\Man)}
        \Big) \| w \|_{L^2(\Man)}    \\ 
     & \lesssim  \|  q \|_{L^2(\Man)}  \| w \|_{L^2(\Man)},     
\end{aligned}
\end{equation} 
where in the last step, we substituted the expression for $\gamma_j$ from \eqref{eq:ci} and the formula for $T$ given in \eqref{eq:T}, in conjunction with the estimate \eqref{eq:esti_H-1}. It follows from \eqref{eq:esti_w} that 
\[ \| w \|_{L^2(\Man)}^2 \eqsim   \int_\Man w^2 \Vol_g  \lesssim  \|  q \|_{L^2(\Man)}^2. \]
Together with \eqref{eq:v_esti} and \eqref{eq:pv_esti}, we arrive at
\begin{equation}\label{eq:pf_Korn_2}
\|\tilde v\|_{L^2(\Man)} + \|P \tilde v\|_{H^1(\Man)} \lesssim \|q\|_{L^2(\Man)}.
\end{equation}

\noindent{\it Step 4.} Now, substituting $q = \mathrm{d}r \odot \mathrm{d}v$ in \eqref{eq:drdv} and it follows from 
\eqref{eq:rigid} that 
\[
(\tilde{v} - v) \in \Ker({\rm D}_r) = \RM[r].
\]
Since $v \in (\RM[r])^{\perp}$, it follows that
\[ (\tilde v - v, \mu)_{\Man} = (\tilde v, \mu)_{\Man}, \quad \forall \mu \in \RM[r]. \]
In other words, $(\tilde v - v)$ is the $L^2$-orthogonal projection of $\tilde v$ onto $\RM[r]$, and therefore
\[ \| \tilde v - v \|_{L^2(\Man)} \leq  \| \tilde  v \|_{L^2(\Man)} \lesssim \|q\|_{L^2(\Man)}= 
\|{\rm D}_r v\|_{L^2(\Man)}. \]
Now, by the norm equivalence in finite dimensional space $\RM[r]$, we arrive at
\begin{equation}\label{eq:pf_Korn_3}
\| \tilde v - v \|_{L^2(\Man)} + \| P (\tilde v - v) \|_{H^1(\Man)} 
\eqsim  \| \tilde v - v \|_{L^2(\Man)}  
\lesssim \|{\rm D}_r v\|_{L^2(\Man)}.
\end{equation} 
Combining this with \eqref{eq:pf_Korn_2}, it follows that
\[
\|v\|_{L^2(\Man)} + \|Pv\|_{H^1(\Man)} 
\lesssim \|{\rm D}_r v\|_{L^2(\Man)}.
\]
The proof is completed by a standard density argument.

\end{document}